\documentclass[11pt]{article} 

\usepackage{amsfonts}
\usepackage{amsmath}
\usepackage{amssymb}
\usepackage{color}
\usepackage[dvipsnames]{xcolor}
\usepackage{graphicx}
\usepackage{latexsym}
\usepackage{mathrsfs} 
\usepackage{multirow}
\usepackage[round]{natbib}
\usepackage{amsthm}

\newcommand{\bel}{\begin{eqnarray}\label}
\newcommand{\eel}{\end{eqnarray}}
\newcommand{\bes}{\begin{eqnarray*}}
\newcommand{\ees}{\end{eqnarray*}}
\newcommand{\bei}{\begin{itemize}}
\newcommand{\eei}{\end{itemize}}
\newcommand{\beiftnt}{\begin{itemize}\footnotesize}

\def\benu{\begin{enumerate}}
\def\eenu{\end{enumerate}}

\def\real{{\mathbb{R}}}
\def\R{{\real}}

\def\E{{\mathbb{E}}}

\def\P{{\mathbb{P}}}

\def\complex{\mathop{{\rm I}\kern-.58em\hbox{\rm C}}\nolimits}

\def\Cov{\hbox{\rm Cov}}
\def\Corr{\hbox{\rm Corr}}
\def\Var{\hbox{\rm Var}}

\def\mathbold{\boldsymbol} 

\def\fhat{\widehat{f}}

\def\calF{{\cal F}}\def\scrF{{\mathscr F}}

\def\calI{{\cal I}}

\def\calJ{{\cal J}}

\def\calT{{\cal T}}

\def\bu{\mathbold{u}}

\def\bX{\mathbold{X}}

\def\bY{\mathbold{Y}}

\def\lam{\lambda}

\newtheorem{theorem}{Theorem}

\newtheorem{lemma}{Lemma}

\newtheorem{corollary}{Corollary}

\def\Corr{\hbox{\rm Corr}}
\def\pen{\hbox{\rm Pen}}

\def\rholam{\rho}

\begin{document}
\title{
Extreme Nonlinear Correlation for Multiple Random Variables and Stochastic 
Processes with Applications to Additive Models
}
\author{Zijian Guo and Cun-Hui Zhang}

%
%
%
%
%
%
%


\centerline{\large \bf 
Extreme Eigenvalues of Nonlinear Correlation {Matrices}} 
\centerline{\large \bf
 with Applications to Additive Models}

\medskip
\begin{center}
{\sc Zijian Guo\footnote{Research partially supported by the NSF DMS-1811857, DMS-2015373.} 
and 
Cun-Hui Zhang\footnote{Research partially supported by the NSF Grants DMS-1513378,
DMS-1721495, IIS-1741390 and CCF-1934924.}
}\\ 
Rutgers University
\end{center}

\begin{center}
\em 
Dedicated to the memory of Larry Shepp
\end{center}

%

\begin{abstract}

The maximum correlation of functions of a pair of random variables is an important measure of stochastic dependence. It is known that this maximum nonlinear correlation is identical to the absolute value of the Pearson correlation for a pair of Gaussian random variables or a pair of finite sums of iid random variables. This paper extends these results to pairwise Gaussian vectors and processes, nested sums of iid random variables, and permutation symmetric functions of sub-groups of iid random variables. It also discusses applications to additive regression models.

\end{abstract}


\section{Introduction} 
The maximum correlation of functions of a pair of random variables 
is an important measure of {stochastic} dependence. Formally, given 
random variables $X_1$ and $X_2$, 
the maximum correlation is defined as 
\bel{NL-corr-pair}
R(X_1,X_2)=\sup\Big\{\Cov\Big( f_1(X_1), f_2(X_2)\Big): 
\Var\big(f_1(X_1)\big)=\Var\big(f_2(X_2)\big) = 1\Big\},
\eel
where $f_1$ and $f_2$ are real functions. 
If $X_1$ and $X_2$ are bivariate normal, it was established in 
{\cite{lancaster1957some}} 
that 
\begin{equation}
R(X_1,X_2)=|\rho(X_1,X_2)|
\label{eq: known}
\end{equation}
where $\rho(X_1,X_2)$ denotes the Pearson correlation between $X_1$ and $X_2$. 
Dembo, Kagan and Shepp (2001) showed that the equality \eqref{eq: known} holds 
{with $R(X_1,X_2) = \sqrt{m/n}$, $1\le m\le n$,} if $X_1$ and $X_2$ are 
respectively nested sums of $m$ and $n$ 
independent and identically distributed (iid) random variables with finite second moment. 
{Following their work, \cite{bryc2005maximum} 
removed the second moment condition for the nested sums, and 
\cite{yu2008maximal} extended the result to two sums of arbitrary finite subsets of iid random variables.}

The current paper extends the above results to more than two random variables and Gaussian processes. 
Let $\lam_{\min}$ and $\lam_{\max}$ denote the smallest and largest  eigenvalues of matrices or  
linear operators, and $\Corr_{\neq}(X_1,\ldots,X_p)$ the $p\times p$ off-diagonal correlation matrix of $p$ random variables with elements $\rho(X_j,X_k)I_{\{j\neq k\}}$. 
Since the maximum correlation of a pair of random variables can be expressed as
$R(X_1,X_2)=\sup_{f_1,f_2}\lambda_{\max}\left(\Corr_{\neq}(f_1(X_1), f_2(X_2)\right),$ 
a natural extension of the maximum nonlinear correlation to the multivariate setting is the 
extreme eigenvalue of the off-diagonal correlation matrix of marginal function transformations of $X_1,\ldots,X_p$, 
\bel{NL-corr-max}
{\rho^{NL}_{\max}}(X_1,\ldots,X_p)  
=\sup_{f_1,\ldots,f_p} \lambda_{\max}\left(\Corr_{\neq}\left(f_1(X_1),\dots,f_p(X_p) \right)\right), 
\eel
where the supreme is taken over all deterministic $f_j$ with $0<\Var\big(f_j^2(X_j)\big)<\infty$, and 
similarly 
\bel{NL-corr-min}
{\rho^{NL}_{\min}}(X_1,\ldots,X_p)  
=\inf_{f_1,\ldots,f_p} \lambda_{\min}\left(\Corr_{\neq}\left(f_1(X_1),\dots,f_p(X_p) \right)\right). 
\eel
For {$p=2$, $\rho^{NL}_{\max} = -\rho^{NL}_{\min}\in [0,1]$. 
However, 
for $p\ge 3$,  
$\rho^{NL}_{\min}\in [-1,0]$ is no longer determined by $\rho^{NL}_{\max}\in [0,p-1]$,} 
so that both quantities are needed to capture the extreme eigenvalues of 
{the off-diagonal nonlinear  correlation matrix}. 
Moreover, {\eqref{NL-corr-max} and \eqref{NL-corr-min} lead 
to the following further extension} 
to stochastic processes: For {any process} $X_{\calT} = \{X_t, t\in \calT\}$ 
on an index set $\calT$ equipped with a measure $\nu$ 
{and $W_{s,t}\ge 0$ as a weight function on $\calT\times\calT$,}  
\bel{NL-corr-max-stoch}
\rholam^{NL}_{\max} 
&=& \rholam^{NL}_{\max}(X_{\calT},{\nu,W})  
\\ \nonumber 
&=& \sup_{f_{\calT}\in\calF_{\calT}}\sup_{\|h\|_{L_2(\nu)}=1} 
\int_{t\in \calT}\int_{s\in \calT }  \rho\left(f_s(X_s),f_t(X_t)\right)
{W_{s,t}}h(s)h(t)\nu(ds)\nu(dt), 
\eel
where $\|h\|_{L_2(\nu)}=\big\{\int_{\calT} h^2(t)\nu(dt)\big\}^{1/2}$ 
{and} $\calF_{\calT}$ is the class of all deterministic 
$f_{\calT}=\{f_t,t\in \calT\}$ 
satisfying proper measurability and integrability conditions. Correspondingly, 
\bel{NL-corr-min-stoch}
\rholam^{NL}_{\min} 
&=& \rholam^{NL}_{\min}(X_{\calT},{\nu,W})  
\\ \nonumber 
&=& \inf_{f_{\calT}\in\calF_{\calT}}\inf_{\|h\|_{L_2(\nu)}=1} 
\int_{t\in \calT}\int_{s\in \calT } \rho\left(f_s(X_s),f_t(X_t)\right){W_{s,t}}h(s)h(t)\nu(ds)\nu(dt). 
\eel
Clearly, \eqref{NL-corr-max} and \eqref{NL-corr-min} are respectively special cases of  
\eqref{NL-corr-max-stoch} and \eqref{NL-corr-min-stoch} with $\calT=\{1,\ldots,p\}$, 
{$W_{s,t}=I_{\{s\neq t\}}$ 
and the counting measure $\nu(A)=|A|$. 
{We refer to 
\eqref{NL-corr-max-stoch} and \eqref{NL-corr-min-stoch} 
as the maximum, minimum or extreme nonlinear correlations of the process $X_{\calT}$. }
Let $K_{W,f_{\calT}}(s,t) = \rho\big(f_s(X_s),f_t(X_t)\big)W_{s,t}$ as a kernel and 
$K_{W,f_{\calT}}: h\to \int K_{W,f_{\calT}}(\cdot,s) h(s)\nu(ds)$ as a linear operator in $L_2(\nu)$. 
The extreme nonlinear correlations in \eqref{NL-corr-max-stoch} and \eqref{NL-corr-min-stoch} 
are expressed as the extreme eigenvalues of the operator $K_{W,f_{\calT}}:L_2(\nu)\to L_2(\nu)$ via  
\bel{extreme-NL-kernel}
\rholam^{NL}_{\max}
= \sup_{f_{\calT}\in\calF_{\calT}} \lambda_{\max}(K_{W,f_{\calT}}),\quad 
\rholam^{NL}_{\min}
= \inf_{f_{\calT}\in\calF_{\calT}} \lambda_{\min}(K_{W,f_{\calT}}). 
\eel
Because the weight function $W$ is almost completely general, it can be used to absorb 
{the Radon-Nikodym derivative between two choices of the measure $\nu$ as follows.} 
The pair $\{\nu',W')$ would yield the same extreme nonlinear correlations as $\{\nu,W)$ when 
the measures $\nu$ and $\nu'$ are absolutely continuous with respect to each other and 
$W'_{s,t} = W_{s,t}\sqrt{\nu(ds)/\nu'(ds)}\sqrt{\nu(dt)/\nu'(dt)}$. 
{For $\calT=\{1,\ldots,p\}$, we may take $\nu$ as the counting measure 
without loss of generality, so that the quantities in \eqref{extreme-NL-kernel} are given by the 
extreme eigenvalues of the matrix $(K_{W,f_{\calT}}(j,k))_{p\times p}$.}
}

The main assertion of this paper is that in a number of settings, 
{the above weighted extreme nonlinear correlations are identical to their} linear counterpart: 
\bel{theme}
\rholam^{NL}_{\max} = \rholam^{L}_{\max}\ \hbox{ and }\ \rholam^{NL}_{\min} = \rholam^{L}_{\min}, 
\eel
where $\rholam^{L}_{\max}$ and $\rholam^{L}_{\min}$ are defined by restricting 
the {functions} 
$f_t$ in \eqref{NL-corr-max-stoch} and \eqref{NL-corr-min-stoch} to be the identity $f(x)=x$; 
e.g. in the more general stochastic process setting, 
\bel{L-corr-max-stoch}
\rholam^{L}_{\max} = \rholam^{L}_{\max}(X_{\calT},{\nu,W})  
=\sup_{\|h\|_{L_2(\nu)}=1} \int_{t\in \calT}\int_{s\in \calT}  \rho\left(X_s,X_t\right){W_{s,t}}h(s)h(t)\nu(ds)\nu(dt), 
\eel
and 
\bel{L-corr-min-stoch}
\rholam^{L}_{\min} = \rholam^{L}_{\min}(X_{\calT},{\nu,W})  
=\inf_{\|h\|_{L_2(\nu)}=1} \int_{t\in \calT}\int_{s\in \calT } 
\rho\left(X_s,X_t\right){W_{s,t}}h(s)h(t)\nu(ds)\nu(dt). 
\eel
{We note that for $W_{s,t}=1$, $\rholam^{NL}_{\max}\le \nu(\calT)$ 
and {$\rholam^{NL}_{\min}\ge 0$}, so that \eqref{theme} is trivial when $\rholam^{L}_{\max} =\nu(\calT)$ 
and $\rholam^{L}_{\min}=0$. In fact, the first identify of \eqref{theme} is nontrivial when $\rholam^{L}_{\max}<\nu(\calT)$  and the second identify of \eqref{theme} is nontrivial when $\rholam^{L}_{\min}>0.$   {However, for general $W_{s,t}$, 
there is no explicit formula for such attainable extreme solutions 
when the maximum and minimum are also taken over all correlation 
operators $\rho(X_s,X_t)$.} Similar to \eqref{extreme-NL-kernel}, we define 
\bel{extreme-L-kernel}
\rholam^{L}_{\max}
= \lambda_{\max}(K_{W}),\quad 
\rholam^{L}_{\min} = \lambda_{\min}(K_{W}). 
\eel
where $K_{W}: h\to \int K_{W}(\cdot,s) h(s)\nu(ds)$ is the linear operator in $L_2(\nu)$ with 
the kernel $K_{W}(s,t) = \E\big[X_sX_t\big]W_{s,t}$. 
As discussed below \eqref{extreme-NL-kernel}, for $\calT=\{1,\ldots,p\}$ we may take 
$\nu$ as the counting measure without loss of generality, so that \eqref{extreme-L-kernel} 
is given by the extreme eigenvalues of the matrix $(K_{W}(j,k))_{p\times p}$. 
}

We will begin by proving \eqref{theme} for Gaussian processes $X_{\calT}$ on an arbitrary 
index set {$\calT$.} 
Our analysis bears some resemblance to that of \cite{lancaster1957some} through the use of 
the Hermite polynomial expansion, but the general functional nature of our problem 
requires additional elements involving the spectrum boundary of the Schur product of linear operators. 
In fact, we prove that only a pairwise bivariate Gaussian condition is required for 
\eqref{theme} under proper measurability and integrability conditions.    


\subsection{Hidden pairwise Gaussian and additive models}
{We generalize the results in \eqref{theme} from pairwise Gaussian vectors to more general random vectors and then present two implications to the analysis of additive models.}
We shall say that a random {vector $X_{1:p}=(X_1,\ldots,X_p)$ is} {\it hidden Gaussian} if 
$X_j = T_j(Z_j)$ for a Gaussian vector ${Z_{1:p}}=(Z_1,\ldots,Z_p)$ and 
some deterministic transformations $T_j, 1\le j\le p$; {$X_{1:p}$ is} {\it hidden pairwise Gaussian} 
if the Gaussian requirement on {$Z_{1:p}$} is reduced to pairwise Gaussian. 
The {identities} in \eqref{theme} for the pairwise 
Gaussian process 
{is equivalent to 
\bes
\rholam^L_{\min}(Z_{1:p},\nu,W) \le \rholam^{NL}_{\min}(X_{1:p},\nu,W),\quad 
\rholam^{NL}_{\max}(X_{1:p},\nu,W) \le\rholam^L_{\max}(Z_{1:p},\nu,W),
\ees
for all measures $\nu$ and weights $W_{s,t}$.} 
That is to say, if the correlation structure of {$X_{1:p}$} is generated from a pairwise Gaussian distribution through marginal transformations, 
{then their extreme nonlinear correlations are controlled within 
the extreme linear correlations} of the underlying {pairwise} Gaussian distribution. 
When {$Z_{1:p}$ is} jointly Gaussian and the transformations $T_j$ are monotone, 
this is the Gaussian copula model widely used in financial risk assessment and other areas of applications. 

Our interest in the extreme nonlinear correlations arises from our study of 
the additive regression model where the response variable {$Y$} can be written as 
$${Y}=\sum_{j=1}^{p} f_{j}(X_j)+\epsilon.$$  
As an important nonlinear relaxation of the linear regression, this model 
effectively mitigates the curse of dimensionality in the more complex 
multiple nonparametric regression 
\citep{buja1989linear,wood2017generalized,hastie1986generalized}. 
Let $\|f\|_{L_2^{(0)}(\P)}$ denote the semi-norm given 
by $\|f\|_{L_2^{(0)}(\P)}^2=\Var(f({X_{1:p}}))$. 
Our result on the minimum eigenvalue of {the} nonlinear correlation matrix 
has two interesting implications in the analysis of 
additive models as follows. 
Firstly, the characterization of $\rho^{NL}_{\min}$ in the current paper 
{can} be used to verify the 
theoretical restricted eigenvalue and compatibility conditions 
required for the analysis of additive models. 
In particular, the {theoretical restricted eigenvalue and compatibility conditions} on the design are critical for establishing {upper} bounds on the prediction error 
{$\| \sum_{j=1}^{p} \fhat_j - \sum_{j=1}^{p} f_j \|_{L_2^{(0)}(\P)}^2$} of regularized estimators $\widehat{f}$ in the additive model   \citep{meier2009high,koltchinskii2010sparsity,raskutti2012minimax, suzuki2013fast,tan2017penalized}. 
Secondly, when the minimum nonlinear correlation of 
{$X_{1:p}$} is bounded away from zero, the squared loss for the estimation of 
individual {$f_j$} can be derived from the prediction error bound via 
\begin{equation*}
{\sum_{j=1}^{p} \|\fhat_j-f_j\|_{L_2^{(0)}(\P)}^2 
\leq \frac{1}{\rholam^{NL}_{\min}}\bigg\| \sum_{j=1}^{p} \fhat_j - \sum_{j=1}^{p} f_j \bigg\|_{L_2^{(0)}(\P)}^2} 
\end{equation*}
{where $\rholam^{NL}_{\min}$ is defined in \eqref{NL-corr-min-stoch} with the counting measure $\nu(A)=|A|$ and uniform weight $W_{s,t}=1$.} 
See Section \ref{sec: comp} for more detailed discussions.

\subsection{Symmetric functions}

In addition to the extension of \cite{lancaster1957some} to pairwise Gaussian processes and vectors, 
the current paper directly extends 
{the results of Dembo, Kagan and Shepp (2001), \cite{bryc2005maximum}} 
{and \cite{yu2008maximal}} 
by establishing \eqref{theme} for 
{nested} sums $\left(X_1,X_2,\cdots, X_{p}\right)$ of iid random variables $Y_i$, 
{with} $X_j=\sum_{i=1}^{m_j} Y_i$ 
for some positive integers {$m_1<\cdots<m_p$.} 
Moreover, as a natural generalization of the {nested} sums, 
we consider groups of the iid variables as random vectors $\bX_j=(Y_{i}, i\in G_j)$ 
where $G_j$ are sets of positive integers. We extend the first part of \eqref{theme} by proving that 
{for the counting measure $\nu$ and any weights $W_{j,k}\ge 0$}
\bel{new-max}
\max_{{\rm symmetric} \; f_1,\ldots,f_p} \rholam_{\max}^{L}\big({\big(f_1(\bX_1),\ldots,f_p(\bX_p)\big),\nu,W}\big) 
= \rholam_{\max}^{L}\big({\big(S_{G_1},\ldots,S_{G_p}\big),\nu,W}\big)  
\eel
{where $S_{G_j} = \sum_{i\in G_j} h_0(Y_i)$ for any deterministic function $h_0$ satisfying 
$0<\Var(h_0(Y_i))<\infty$} and the maximum is taken 
{over} all deterministic functions $f_i$ symmetric in the permutation of its arguments. 
{In the sequel, such $f_i$ are simply called symmetric functions.}  
We also establish the corresponding {identity for the minimum correlation,} 
\bel{new-min}
\min_{{\rm symmetric} \; f_1,\ldots,f_p} \rholam_{\min}^{L}
\big({\big(f_1(\bX_1),\ldots,f_p(\bX_p)\big),\nu,W}\big) 
= \rholam_{\min}^{L}\big(\big({S_{G_1},\ldots,S_{G_p}\big),\nu,W}\big),  
\eel
under a mild condition {which holds when} 
$\cap_{j=1}^p G_j\neq \emptyset$. 


\subsection{Paper organization}

The rest of the paper is organized as follows. In Section \ref{sec: Gaussian}, we study the extreme eigenvalues of nonlinear correlation matrix for pairwise Gaussian random {vectors and processes;} 
In Section \ref{sec: comp}, we discuss the implications {of our results in Section \ref{sec: Gaussian} on} 
additive models; In Section \ref{sec: symiid}, we study the extreme eigenvalues of nonlinear correlation matrix of {nested} sums and also the more general symmetric functions {of iid random variables.}  

\section{
Pairwise Gaussian Processes}
\label{sec: Gaussian}
%
%
To start with, we shall explicitly specify the measurability and integrability conditions 
for the definition of the extreme linear and nonlinear correlations in 
\eqref{L-corr-max-stoch}, \eqref{L-corr-min-stoch},
\eqref{NL-corr-max-stoch} and \eqref{NL-corr-min-stoch}. 

\noindent {\bf Assumption A:} (i) {\it 
{There exist $B_n\subset B_{n+1}\subset \calT$ such that $\cup_{n=1}^\infty B_n=\calT$, 
$\nu(B_n)<\infty$ and 
$\int_{B_n}\int_{B_n} W^2_{s,t}\nu(ds)\nu(dt)<\infty$ for any positive integer $n\geq 1$.}\\ 
(ii) {\it The process $X_{\calT}$ is standardized to $\E[X_t]=0$ and $\E[X_t^2]=1$, 
the {correlation operator $\E\big[X_sX_t\big]$} is measurable as a function of 
$(s,t)$ in the product space $\calT\times\calT$, and 
the weight function $W_{s,t}$ is element-wise nonnegative 
{and symmetric, 
$W_{s,t}=W_{t,s}\ge 0$.}}\\ 
(iii) {\it {The operator $K_W$ in \eqref{extreme-L-kernel} is bounded. }}
}

We note that there is no loss of generality to assume that $X_{\calT}$ is standardized 
as \eqref{L-corr-max-stoch} and \eqref{L-corr-min-stoch} involve only the correlation between 
$X_s$ and $X_t$. 
Under Assumption A {(iii),} the operator $K_W$ 
yielding finite extreme linear correlations in 
\eqref{L-corr-max-stoch} and \eqref{L-corr-min-stoch}. 


\noindent {\bf Assumption B:} {\it In \eqref{NL-corr-max-stoch} and \eqref{NL-corr-min-stoch},  
$\calF_{\calT}$ is the class of all function families $f_{\calT}=\{f_t,t\in \calT\}$ 
with $\E[f_t(X_t)]=0$, $\E[f_t^2(X_t)]>0$ and $\int_{\calT}\E[f_t^2(X_t)]\nu(dt)<\infty$ such that 
$\E\big[X_t^m f_t(X_t)\big]$ are measurable functions of $t$ on $\calT$ for all {integers} 
$m\ge 1$, and 
in \eqref{extreme-NL-kernel} the kernel $K_{W,f_{\calT}}(s,t)=\Corr(f_s(X_s),f_t(X_t))W_{s,t}$ 
is a measurable function of $(s,t)$ on {$\calT\times\calT$.}} 
 
{In} the discrete case where $\calT = \{1,\ldots,p\}$, 
Assumption A {always holds} when $\E[X_t]=0$ and $\E[X_t^2]=1$ and 
Assumption B {always holds} when $\calF_{\calT}$ {is the set of} all $f_{\calT} = \{f_1,\ldots,f_p\}$ 
satisfying $\E[f_j(X_j)]=0$ and $0<\E[f_j^2(X_j)]<\infty$, $j=1,\ldots,p$. 

We first establish some equivalent expressions 
to \eqref{NL-corr-max-stoch} and \eqref{NL-corr-min-stoch}
in the following lemma.
\begin{lemma}\label{lm-1}
Let $\rholam^{NL}_{\max}$ and $\rholam^{NL}_{\min}$ be as in 
\eqref{NL-corr-max-stoch} and \eqref{NL-corr-min-stoch} 
with the function class $\calF_{\calT}$ specified in Assumption B. 
Then, 
\begin{equation}
\rholam^{NL}_{\max} =\sup_{f_{\calT}\in\calF_{\calT}} 
\frac{ \int_{t\in \calT}\int_{s\in \calT }  \E\big[f_s(X_s),f_t(X_t)\big] {W_{s,t}} \nu(ds)\nu(dt)}
{\int_{t\in \calT} \E\big[f_t^2(X_t)\big] \nu(dt)},
\label{eq: multi extreme}
\end{equation}
and 
\begin{equation}
\rholam^{NL}_{\min} =\inf_{f_{\calT}\in\calF_{\calT}} 
\frac{ \int_{t\in \calT}\int_{s\in \calT } \E\big[f_s(X_s),f_t(X_t)\big] {W_{s,t}} \nu(ds)\nu(dt)}
{\int_{t\in \calT} \E\big[f_t^2(X_t)\big] \nu(dt)}. 
\label{eq: multi min extreme}
\end{equation}
\label{prop: equi express}
\end{lemma}

A proof of {Lemma \ref{lm-1}} can be found in the Appendix.
{The more explicit expressions established in the lemma} 
{will} facilitate the Hermite {polynomial} expansion of the covariance {in our analysis.} 
Another ingredient of our analysis, stated in the following lemma, 
concerns the extreme eigenvalues of the Schur {product.} 


\begin{lemma}\label{lm-2}
Let $\rholam^{L}_{\max}$ and $\rholam^{L}_{\min}$ be as in 
\eqref{L-corr-max-stoch} and \eqref{L-corr-min-stoch} respectively 
and {$K_W(s,t)=\E[X_sX_t]W_{s,t}$.} Under Assumption A, 
\bel{prop: power eigen} 
\rholam^{L}_{\min} \le \int_{t\in \calT}\int_{s \in \calT} 
{\big(\E[X_sX_t]\big)^{m-1}K_W(s,t)} h(s)h(t)\nu(ds)\nu(dt) \le \rholam^{L}_{\max}. 
\eel
for any integer {$m\ge 1$} and function $h(t)$ with $\int h^2(t)\nu(dt)=1$. 
\end{lemma}

The above lemma establishes that the spectrum of the operator given by the Schur power kernel 
${\big(\E[X_sX_t]\big)^{m-1}K_W(s,t) = \big(\E[X_sX_t]\big)^{m}W_{s,t}}$ 
is controlled inside that of {$K_W(s,t)$, so that the Schur 
{multiplication} of a correlation matrix is a contraction.}  
The proof of the {lemma,}  
given in the Appendix, utilizes an interesting construction of the Schur power 
kernel with iid copies of $X_{\calT}$. 
Such a proof technique is {simple but quite useful.} 

We are now ready to state the equivalence between the extreme nonlinear correlation 
and the extreme linear correlation for pairwise Gaussian processes. 

\begin{theorem}\label{th-1}
Let $X_{\calT}=\{X_t\}_{t\in\calT}$ be a pairwise Gaussian process in the sense that 
$(X_s,X_t)$ are bivariate Gaussian vectors for all pairs $(s,t)\in \calT\times\calT$. 
Under Assumptions A and B, 
\bes
\rholam^{NL}_{\max} = \rholam^{L}_{\max}\ \hbox{ and }\ \rholam^{NL}_{\min} = \rholam^{L}_{\min}, 
\ees
where $\rholam^{NL}_{\max}$ and $\rholam^{NL}_{\min}$ are defined
in \eqref{NL-corr-max-stoch} and \eqref{NL-corr-min-stoch} respectively, and 
$\rholam^{L}_{\max}$ and $\rholam^{L}_{\min}$ are defined
in \eqref{L-corr-max-stoch} and \eqref{L-corr-min-stoch} respectively.   
\end{theorem}


\begin{proof} 
As the normalized Hermite polynomials 
\bes
H_m(x) = (m!)^{-1/2} (-1)^m e^{x^2/2}(d/dx)^m e^{-x^2/2} 
\ees
form a orthonormal system with $\E[H_m(Z)]=0$ and $\E[H^2_m(Z)]=1$ for 
$Z\sim N(0,1)$, by Assumptions A and B we may write 
$f_t(X_t) = \sum_{m=1}^\infty a_m(t)H_m(X_t)$ in the sense of $L_2$ convergence.  Let 
{$K_W^m(s,t) = \big(\E\big[X_s,X_t\big]\big)^{m-1} K_W(s,t) = \big(\E\big[X_s,X_t\big]\big)^m W_{s,t}$.}  
As $(X_s,X_t)$ is bivariate normal with $\Var(X_s)=\Var(X_t)=1$, 
$\E[H_m(X_s)H_n(X_t)] {W_{s,t} = {K_W^m}(s,t)} {\bf 1}_{m=n}$ 
{as in \cite{lancaster1957some}.} 
It follows that  
$\E\big[f_s(X_s)f_t(X_t)\big]{W_{s,t}} = \sum_{m=1}^\infty {K_W^m}(s,t) a_m(s)a_m(t)$.  
{As $|K_W^m(s,t)|\le K_W^2(s,t)$, Lemma \ref{lm-2} provides} 
\bes
&& \int_{s\in \calT}\int_{t\in \calT} \E\big[f_s(X_s),f_t(X_t)\big] {W_{s,t}} \nu(ds)\nu(dt)
\cr &=&  \int_{s\in \calT}\int_{t\in \calT}\bigg\{\sum_{m=1}^\infty {K_W^m}(s,t) a_m(s)a_m(t)\bigg\}\nu(ds)\nu(dt)
\cr &\le&  {\int_{s\in \calT}\int_{t\in \calT} K_W(s,t) a_1(s)a_1(t)\nu(ds)\nu(dt)}
\cr && {+ \sum_{m=2}^\infty \int_{s\in \calT}\int_{t\in \calT} K_W^2(s,t)\big|a_m(s)a_m(t)\big|\nu(ds)\nu(dt)} 
\cr & \le & \rholam^{L}_{\max}\sum_{m=1}^\infty \int a_m^2(t)\nu(dt)
\cr & = & \rholam^{L}_{\max} \int_{t\in \calT} \E\big[f_t^2(X_t)\big] \nu(dt). 
\ees
Moreover, as the exchange of summation and integration is allowed as the above, 
\bes
&& \int_{s\in \calT}\int_{t\in \calT} \E\big[f_s(X_s),f_t(X_t)\big]{W_{s,t}} \nu(ds)\nu(dt)
\cr &=& \sum_{m=1}^\infty \int_{s\in \calT}\int_{t\in \calT} \big\{{K_W^m}(s,t) a_m(s)a_m(t)\big\}
\nu(ds)\nu(dt)
\cr & \ge & \rholam^{L}_{\min}\sum_{m=1}^\infty \int a_m^2(t)\nu(dt)
\cr & = & \rholam^{L}_{\min} \int_{t\in \calT} \E\big[f_t^2(X_t)\big] \nu(dt). 
\ees
The proof is complete as inequalities in the other direction are trivial. 
\end{proof}

Theorem \ref{th-1} establishes the equality of the extreme eigenvalues of the nonlinear and linear 
correlation operators. However, as we have mentioned in the introduction, such results could be trivial 
when $\rholam^L_{\max}$ and $\rholam^L_{\min}$ attain the extreme eigenvalues among all correlation 
operators. {In the following three subsections, we discuss 
the discrete case $\calT=\{1,\ldots,p\}$, the continuous case $\calT=[0,1]$, and stationary processes as 
three nontrivial examples and state the implications of Theorem~\ref{th-1}  
as corollaries. }

\subsection{{Hidden pairwise Gaussian vectors}}
The following part demonstrates the application of 
{Theorem \ref{th-1}} to a finite number of pairwise Gaussian random variables, that is, 
{$\mathcal{T}=\{1,2,\cdots,p\}$.  
As we discussed below \eqref{extreme-NL-kernel} and \eqref{extreme-L-kernel}, 
we take $\nu$ as the counting measure without loss of generality throughout the subsection. 
}

{
\begin{corollary}\label{thm: nonlinear eigen}
Let $X_{1},X_{2},\cdots, X_{p}$ be pairwise Gaussian random variables with $X_j\sim N(0,1)$ and 
a correlation matrix $\Sigma = (\Sigma_{j,k})_{p\times p}$. 
Let $W=(W_{j,k})_{p\times p}$ be a matrix with elements $W_{j,k}=W_{k,j}\ge 0$ 
and $\Sigma\circ W =(\Sigma_{j,k}W_{j,k})_{p\times p}$ be the Schur product. 
Then, for all functions $f_j$ satisfying $\E f_j(X_j)=0$ and $0<\E f^2(X_j)<\infty$, 
\begin{equation}
\lambda_{\min}\left(\Sigma\circ W\right)
\leq \frac{\E \left[\sum_{j=1}^{p}\sum_{k=1}^p W_{j,k}f_{j}(X_{j})f_{j}(X_{k})\right]}{\sum_{j=1}^{p}\E f^2_{j}(X_{j})}
\leq \lambda_{\max}\left(\Sigma\circ W\right).
\label{eq: bound Gaussian-W}
\end{equation}
In particular, for $\Sigma\circ W = \Sigma$ with $W_{j,k}=1$, 
\begin{equation}
\lambda_{\min}\left(\Sigma\right)\cdot \sum_{j=1}^{p}\E f^2_{j}(X_{j})\leq \E \left(\sum_{j=1}^{p}f_{j}(X_{j})\right)^2\leq \lambda_{\max}\left(\Sigma\right)\cdot \sum_{j=1}^{p}\E f^2_{j}(X_{j}).
\label{eq: bound Gaussian}
\end{equation}
Equivalently, for $W_{j,k}=I_{\{j\neq k\}}$, 
\eqref{NL-corr-max} and \eqref{NL-corr-min} 
are given by their linear version, 
\begin{equation}
\rho_{\max}^{NL}(X_1,\ldots,X_p) =\lambda_{\max}(\Sigma)-1
\quad \text{and}\quad 
\rho_{\min}^{NL}(X_1,\ldots,X_p) = \lambda_{\min}(\Sigma) - 1. 
\label{eq: result 1a}  
\end{equation}
\end{corollary}
}

In the setting of the above corollary, 
{the operator $K_W$ in \eqref{extreme-L-kernel} is given by the Schur product matrix 
$K_W=\Sigma \circ W$, and for general weights $W$
\eqref{eq: bound Gaussian} and \eqref{eq: result 1a} are nontrivial with 
$\lambda_{\min}(\Sigma)>0$ and $\lambda_{\max}(\Sigma)<p$ 
when $\Sigma$ is of full rank.}  

%

Finally, we state in the following corollary the implication of Theorem \ref{th-1} on 
Gaussian copula and other hidden pairwise Gaussian variables. 

\begin{corollary}
Suppose ${X_{1:p} =} \left(X_{1},X_{2},\cdots, X_{p}\right)$ follows a hidden Gaussian distribution 
in the sense of $X_j =T_j(Z_j)$ for a Gaussian 
vector {$Z_{1:p} = (Z_1,\ldots,Z_p)$} and 
some deterministic functions $T_j$ with $0<\Var(T_j(Z_j))<\infty$. 
{Let $\Sigma^z$ be the covariance matrix of the hidden vector $Z_{1:p}$. 
Then, for the counting measure $\nu$ and any symmetric $W$ with $W_{j,k}\ge 0$,
\begin{equation*}
 \lambda_{\min}(\Sigma^z\circ W) 
 \leq \rholam_{\min}^{NL}(X_{1:p},\nu,W),\quad 
 \rholam_{\max}^{NL}(X_{1:p},\nu,W)  
 \leq \lambda_{\max}(\Sigma^z\circ W), 
 \label{eq: result 1b}
\end{equation*}
and the above inequalities become equality when $T_j$ are almost surely invertible. 
In particular, \eqref{eq: bound Gaussian} holds with $\Sigma$ replaced by $\Sigma^z$.} 
Moreover, the Gaussian assumption on {$Z_{1:p}$} can be weakened to pairwise Gaussian.  
\label{cor: nonlinear correlation hidden}
\end{corollary}

Similarly to Corollary \ref{thm: nonlinear eigen}, the upper and lower bounds in the above corollary 
{are nontrivial when the covariance matrix of $Z_{1:p}$ is of full rank.} 
  The above corollary has interesting implications as it states that the extreme eigenvalues of nonlinear correlation matrix {fall into the spectrum range} of the covariance matrix of the underlying generating Gaussian distribution. This is meaningful in statistical applications, that is, the well conditioning of the covariance matrix of general nonlinear transformations follows from that of the underlying generating Gaussian covariance matrix.

\subsection{{Processes on finite intervals}}

Our result for a general pairwise Gaussian process with general index set also directly 
leads to the same for Gaussian process on finite intervals. 
As discussed below \eqref{extreme-NL-kernel}, we take $\calT=[0,1]$ and 
the Lebesgue measure $\nu(dt)=dt$ without much loss of generality. 

\begin{corollary}\label{cor-1}
Let $\{X_t, 0 \le t\le 1\}$ be a Gaussian process with correlation $\rho(X_s,X_t)$ 
and {$W_{s,t}$} be a nonnegative symmetric square integrable function of $(s,t)$ in $[0,1]^2$. 
Let $K_W(s,t)=\rho(X_s,X_t)W_{s,t}$. 
Let $K_W$ be the linear operator $h(\cdot) \to \int_0^1 K_W(\cdot,s)h(s)ds$. 
Then, 
\bes
\rholam^{NL}_{\max} 
= \lambda_{\max}\left(K_W \right), \quad 
\rholam^{NL}_{\min} = \lambda_{\min}\left(K_W \right), 
\ees
for the extreme nonlinear correlations in \eqref{NL-corr-max-stoch} and \eqref{NL-corr-min-stoch},   
where $\lambda_{\max}\left(K_W \right)$ and $\lambda_{\min}\left(K_W \right)$ are 
the extreme eigenvalues of the operator $K_W$ in $L_2([0,1])$. 
In particular, for $K_{1}$ with $W=1$, 
{\bes
\Var\left(\int_0^1 f(X_t,t) dt\right) 
\leq \lambda_{\max}\left(K_1 \right)\int_0^1 \Var\big(f(X_t,t)\big) dt 
\ees}
for all bounded continuous bivariate functions $f(x,t)$. 
\end{corollary}

In the above corollary, {$\lambda^2_{\max}(K_1) 
\le \int_{0}^{1}\int_{0}^{1} \rho^2(X_s, X_t) ds dt \leq 1$}, so that the maximum eigenvalue 
of the correlation operator $K_1$ is nontrivial unless $\rho^2(X_s, X_t)\equiv 1$. 
As the operator $K_1$ is Hilbert-Schmidt, the minimum eigenvalue 
$\lam_{\min}(K_1)=0$ is always trivial. 
However, $\lam_{\min}(K_W)$ may take nontrivial negative values for general $W$. 

The setting here is related to the 
{nested} sum problem considered in Section \ref{sec: symiid} as follows. 
Let $S_j = \sum_{i=1}^j Y_i, 1\le j\le p$, where $Y_i$ are iid random variables with 
$\E[Y_i]=0$ and $\E[Y_i^2]=1$. 
As the correlation kernel of $\{X_t = S_{\lfloor pt\rfloor}/\sqrt{p}, 0\le t\le 1\}$ uniformly 
converges to the correlation kernel $K_1(s,t)=(s\wedge t)/\sqrt{st}$ of the standard Brownian motion 
as $p\to \infty$, 
\bes
\lam_{\max}^L(S_{1:p},\nu^{(p)},W^{(p)}) \to \lam_{\max}(K_W),\quad  
\lam_{\min}^L(S_{1:p},\nu^{(p)},W^{(p)}) \to \lam_{\min}(K_W),
\ees
where $\nu^{(p)}(A) = |A|/p$ is the normalized counting measure 
and {$W^{(p)}_{j,k}=W_{j/p,k/p}$} for $A\cup\{j,k\}\subseteq \{1,\ldots,p\}$, 
and $K_W$ is treated as the operator in $L_2([0,1])$ as in Corollary \ref{cor-1}.

{
\subsection{{Stationary processes}}
In this subsection we consider stationary processes on the entire set of integers and the entire real line. 
In both cases we consider stationary pairwise Gaussian $X_t$ with $\E[X_t]=0$, $\E[X_t^2]=1$ 
and autocorrelation function 
\bes
\rho(s,t) = \rho(t-s) = \E[X_sX_t]. 
\ees
We consider weight functions $W_{s,t} = W_{t,s}=W(t-s)\ge 0$ and write the kernel as 
\bes
{K_W(s,t) = K_W(s-t),\quad  K_W(t) = \rho(t)W(t). }
\ees
We shall fist consider the discrete case. 

\begin{corollary}\label{cor-discrete}
Let $\{X_t, t\in \calT\}$, $\calT = \{0,\pm 1,\pm 2,\ldots\}$, 
be a pairwise Gaussian stationary sequence with autocorrelation $\rho(t)=\E[X_tX_0]$. 
Let $\nu$ be the counting measure on $\calT$, 
$W_{s,t}=W_{t,s}=W(t-s)\ge 0$, $K_W(t)=\rho(t)W(t)$, 
and $K_W: h \to \sum_{s\in\calT} K_W(\cdot-s)h(s)$ as an operator in $\ell_2$. 
Suppose $\sum_{t\in \calT}|K_W(t)|<\infty$. Then, 
\eqref{theme} and \eqref{extreme-L-kernel} hold with   
\bes
\rholam^{NL}_{\max}
= \lambda_{\max}\left(K_W \right) = \sup_{|\omega|\le\pi}\big|K_W^*(\omega)\big|,\quad 
\rholam^{NL}_{\min}
= \lambda_{\min}\left(K_W \right) = \inf_{|\omega|\le\pi}\big|K_W^*(\omega)\big|, 
\ees
where $K_W^*(\omega) = \sum_{s\in\calT}K_W(s)\cos(\omega s)$. 
In particular, for the autoregression sequence with $\rho(t) = \beta^{|t|}$, $|\beta|<1$ necessarily, we have 
$K_1^*(\omega)=(1-\beta^2)/(1+\beta^2- 2\beta \cos(\omega))$, 
$\lambda_{\max}\left(K_1 \right)=(1+|\beta|)/(1-|\beta|)$ and 
$\lambda_{\min}\left(K_1 \right) = (1-|\beta|)/(1+|\beta|)$ for $W(t)=1$. 
\end{corollary}

We note that Corollary \ref{cor-discrete} gives the autoregression and implicitly 
many other examples in which $\nu(\calT)=\infty$ 
and $\lambda_{\max}\left(K_W\right)<\infty$ and $\lambda_{\min}\left(K_W\right)>0$ are both nontrivial. For a pairwise Gaussian stationary process $\{X_t, t\in \calT\}$, the process $\{f_t(X_t), t\in \calT\}$ is in general non-Gaussian and non-stationary, and its spectrum is typically not tractable. 
Still, Corollary \ref{cor-discrete} shows that the spectrum of the nonlinear $\{f_t(X_t), t\in \calT\}$ is contained 
within the spectrum of the underlying process $\{X_t, t\in \calT\}$ under mild conditions. 

\begin{proof} 
Let $F: h\to Fh$ be the mapping from complex $h\in L_2([-\pi,\pi])$ to its Fourier series  
$(Fh)(t) = (2\pi)^{-1/2}\int_{-\pi}^\pi e^{i\omega t} h(\omega)d\omega, t\in\calT$. 
For $h$ with finitely many nonzero coefficients, 
\bes
(K_WFh)(t) 
= \sum_{s\in\calT} K_W(|t-s|)\int_{-\pi}^\pi \frac{e^{i\omega s}}{(2\pi)^{1/2}}h(\omega)d\omega
= \int_{-\pi}^\pi \frac{e^{i\omega t}}{(2\pi)^{1/2}}K_W^*(\omega)h(\omega)d\omega. 
\ees
Let $K_W^*$ be the mapping $h(\omega)\to K_W^*(\omega) h(\omega)$ 
in $L_2([-\pi,\pi])$. As $h$ with finitely many nonzero Fourier coefficients are dense in 
$L_2([-\pi,\pi])$ and $F$ is isometric from $L_2([-\pi,\pi])$ onto $\ell_2$, 
the above calculation implies $K_WF = FK^*_W$. Moreover, 
the spectrum decomposition $K_W=\int \lam dP_\lam$ is given by projections $P_\lam = FP_\lam^*F^{-1}$ where 
$P_\lam^*h(\omega) = h(\omega)I\{K_W^*(\omega)\le\lam\}$ gives the spectrum decomposition 
$K^*_W=\int\lam dP^*_\lam$. 
This gives the main conclusion. For $K_1(t)=\beta^{|t|}$, 
$K_1^*(\omega)=\hbox{$\sum_{s\in\calT}$}\beta^{|s|}e^{i\omega s}
=(1-\beta^2)/(1+\beta^2- 2\beta \cos(\omega))$ 
gives the spectrum. 
\end{proof} 

Next, we consider the continuous case. 
The spectrum of the Ornstein-Uhlenbeck process $K_1(t)=e^{-|t|}$ was studied in \cite{li1992comparison} 
by directly solving the eigenvalue problem for the restriction of the kernel $K_1(s,t)=e^{-|t-s|}$ on 
$L_2([a,b])$ in the proof of Theorem 5 there. 

\begin{corollary}\label{cor-cont}
Let $\{X_t, t\in \R\}$ be a stationary pairwise Gaussian process on the entire real line 
with autocorrelation $\rho(t)=\E[X_tX_0]$. 
Let $\nu$ be the Lebesgue measure on $\R$, 
$W_{s,t}=W_{t,s}=W(t-s)\ge 0$, $K_W(t)=\rho(t)W(t)$, 
and $K_W: h \to \int_{-\infty}^\infty K_W(\cdot-s)h(s)ds$ as an operator in $L_2(\R)$. 
Suppose $\int_{-\infty}^\infty |K_W(t)| dt<\infty$. Then, 
\eqref{theme} and \eqref{extreme-L-kernel} hold with   
\bes
\rholam^{NL}_{\max}= \lambda_{\max}\left(K_W \right) = \sup_{\omega}\big|K_W^*(\omega)\big|,\quad 
\rholam^{NL}_{\min}= \lambda_{\min}\left(K_W \right) = \inf_{\omega}\big|K_W^*(\omega)\big|, 
\ees
where $K_W^*(\omega) = \int_{-\infty}^\infty K_W(s)\cos(\omega s)ds$. 
In particular, for the Ornstein-Uhlenbeck process with $\rho(t) = e^{-|t|}$, 
$K_1^*(\omega)=2/(1+\omega^2)$, 
$\lambda_{\max}\left(K_1 \right)=2$ and 
$\lambda_{\min}\left(K_1 \right) = 0$ for $W(t)=1$. 
\end{corollary}

\begin{proof} 
Here $F$ is the Fourier transformation   
$(Fh)(t) = (2\pi)^{-1/2}\int_{-\infty}^\infty e^{i\omega t} h(\omega)d\omega$, and 
\bes
(K_WFh)(t) 
= \int_{-\infty}^\infty K_W(|t-s|)\int_{-\infty}^\infty \frac{e^{i\omega s}}{(2\pi)^{1/2}}h(\omega)d\omega ds 
= \int_{-\infty}^\infty \frac{e^{i\omega t}}{(2\pi)^{1/2}}K_W^*(\omega)h(\omega)d\omega. 
\ees
This certainly holds for $h\in L_2(\R)\cap L_1(\R)$ which is dense in $L_2(\R)$. 
Again, as $F$ is isometric in $L_2(\R)$, $K_WF = FK^*_W$ with the operator 
$K_W^*: h(\omega)\to K_W^*(\omega) h(\omega)$. 
This gives the spectrum decomposition and spectrum limits of $K_W$ 
as in the proof of Corollary \ref{cor-discrete}. 
For $K_1(t)=e^{-|t|}$, 
$K_1^*(\omega)=\int e^{-|s|}e^{i\omega s}ds =1/(1-i\omega)+1/(1+i\omega)$ 
gives the spectrum. 
\end{proof} 
}

Our problem is also related to mixing conditions on stochastic processes. 
For example, when $\{X_t, -\infty < t <\infty\}$ is a Gauss-Markov process, 
its $\rho$-mixing coefficient, given by 
\bes
{\rho^*(n)} = \sup_t\sup_{f\in \scrF_{(-\infty,t]},g\in \scrF_{[t+n,\infty)}}\Corr(f,g)
\ees
with $\scrF_A$ being the set of all nonzero square integrable functions of $\{X_t, t\in A\}$, 
can be characterized by 
$\rho^*(n) = \sup_t \Corr(X_t,X_{t+n})$ by \eqref{eq: known}. 
However, as this paper is mainly motivated by the application in the additive model 
as discussed in Section \ref{sec: comp} and the multivariate extension of Dembo, Kagan and Shepp (2001) 
as discussed in Section \ref{sec: symiid}, our results do not yield a direct extension of the above explicit 
calculation of the $\rho$-mixing condition to more general processes. 
We refer to \cite{bradley2005basic} 
for a survey of the relationship among different mixing conditions.

 \section{Applications to Additive Models}
 \label{sec: comp}

In this section, we discuss applications of 
{our results to} 
additive models, 
including justification of theoretical restricted eigenvalue and compatibility conditions 
and derivation of convergence rates for the estimation of individual component 
functions from prediction error bounds. 
In the additive regression model, the relationship between the response variable $Y$ 
and design {vector $X_{1:p}=(X_1,\ldots,X_p)$} is given by 
$Y = \sum_{j=1}^{p} f_{j}(X_j)+\varepsilon$, 
{or $\E[Y|X_{1:p}] =\sum_{j=1}^{p} f_{j}(X_j)$ in terms of the conditional expectation,    
where $f_j$ are assumed to be smooth functions and} $\varepsilon$ is the noise variable 
independent of {$X_{1:p}$ with $\E[\varepsilon]=0$.}

Additive models have been important tools for practical data analysis \citep{friedman1981projection,buja1989linear}, mainly due to the fact that it relaxes the stringent 
{model assumption in linear regression} and at the same time {mitigates} 
the curse of dimensionality 
{in multiple nonparametric regression with $Y = f(X_{1:p})+\varepsilon$.} 
Another advantage of {the} additive model is {the natural interpretation of its components. 
For example,} the rate of change of the $j$-th function $f_j$ 
represents the effect of the covariate $X_{j}$ {as in linear regression.} 

Due to its importance, 
{additive models have been extensively} investigated in both the classical low-dimension setting and {the more contemporary high-dimensional setting where only a much smaller number 
than $p$ of the components $f_j$ are actually nonzero. In both cases, one}  
of the main assumptions is the invertibility condition of the additive model. In the very special linear regression setting $Y = \sum_{j=1}^{p} X_j\beta_j+\varepsilon$, this invertibility condition is 
{that} the minimum eigenvalue of the {sample} covariance matrix of {the design 
vector $X_{1:p}$} is bounded away from zero. In the more general additive model, {the component functions $f_j$ are} not necessarily linear and the invertibility condition is 
{typically imposed on the sample covariance 
matrix of certain basis functions of $f_j$.} 
In the following, we connect the invertibility condition to $\rho^{NL}_{\min}$ 
{in \eqref{L-corr-min-stoch}}  
and verify them over a large class of distributions, in both the {low- and high-dimensional settings.}

\subsection{Implications to low-dimensional additive models}
{We start with the low-dimensional setting where the number of covariates 
$p$ is fixed or much smaller than the diverging sample size $n$. 
A useful way of understanding and implementing the additive models 
is to consider the projection of the response to the linear span of suitable bases of the 
component functions $f_j$. 
Denote a set of basis functions for $f_j$ by 
$B_j=B_{j}(x_j)=(B_{j,1} (x_{j})\cdots B_{j, M_{j}} (x_{j}))^{\intercal}\in \R^{M_j}$ with some positive integer $M_{j}$, where the basis can be taken as Fourier, spline, wavelet or other 
constructions and $M_j$ is allowed to grow as the sample size increases. Under proper smoothness conditions, $f_j(x_j)$ can be adequately approximated by 
a linear combination $a_j^{\intercal}B_j(x_j)$ of its basis functions, 
resulting in a $d^*$-dimensional 
regression $E[Y|X_{1:p}]\approx \sum_{j=1}^pa_j^{\intercal}B_j(X_j)$ 
with very large $d^* = \sum_{j=1}^p M_j$. 

When iid copies {$\{(X_{i,1},\ldots,X_{i,p},Y_i)\}_{1\leq i\leq n}$} of $(X_1,\ldots,X_p,Y)$ are observed, 
the invertibility condition in this $d^*$-dimensional regression can be written as 
\begin{equation}
\P\left\{\min_{\sum_{j=1}^p \|\alpha_j\|_2^2 = 1} 
\frac{1}{n}\sum_{i=1}^n\left(\sum_{j=1}^p a_j^{\intercal}B_j(X_{i,j}) 
- \left(\frac{1}{n}\sum_{i=1}^n\sum_{j=1}^p a_j^{\intercal}B_j(X_{i,j})\right)\right)^2 \geq \kappa_0\right\}\to 1 
\label{eq: sample invert}
\end{equation}
with some fixed positive constant $\kappa_0$. 
While probabilistic methods such as empirical process and noncommutative Bernstein inequality 
can be used to verify the above condition, 
such analysis invariably requires the following population invertibility condition: 
$$
\min_{\sum_{j=1}^p \|\alpha_j\|_2^2 = 1} 
\left\|\sum_{j=1}^{p} \alpha^{\intercal}_j B_{j}(X_{j})\right\|_{L_2^{(0)}(\P)}^2 \geq \kappa_0, 
$$
as $\|f(X_{1:p})\|_{L_2^{(0)}(\P)}^2 = \Var(f(X_{1:p}))$ for all functions $f:\R^p\to\R$. 
This population invertibility condition can be decomposed into a 
component-wise invertibility condition 
$$
\left\|\alpha^{\intercal}_j B_{j}(X_{j})\right\|_{L_2^{(0)}(\P)}^2 \geq \kappa_1\|\alpha_j\|_2^2,
 \; \text{for}\; j=1,\ldots, p \; \text{and}\; \kappa_1>0, 
$$
and a population predictive invertibility condition  
\begin{equation}
\left\|\sum_{j=1}^{p} \alpha^{\intercal}_j B_{j}(X_{j})\right\|_{L_2^{(0)}(\P)}^2
\ge \kappa_2\sum_{j=1}^{p}\left\|\alpha^{\intercal}_j B_{j}(X_{j})\right\|_{L_2^{(0)}(\P)}^2 \; \text{for}\; 
\kappa_2>0. 
\label{eq: requirement low}
\end{equation}

Assume that after proper centering and scaling the support of $X_{1:p}$ is $[0,1]^p$. 
The component-wise invertibility condition is fulfilled when 
$B_j$ is orthonormal in $L_2([0,1])$ and the marginal density of $X_j$ is uniformly 
greater than $\kappa_1$ in $[0,1]$. 
The orthonormal condition on $B_j$ can be further weakened to 
$\big\|\alpha^{\intercal}_j B_{j}(X_{j})\big\|_{L_2^{(0)}([0,1])}^2 \gtrsim \|\alpha_j\|_2^2$ 
as in the case of $B$-spline. 
It is also well known that the population predictive invertibility condition 
\eqref{eq: requirement low} holds when the joint density of 
$X_{1:p}$ is uniformly greater than $\kappa_0$ in $[0,1]^p$. 
However, while the lower bound assumption on the individual marginal densities 
approximately holds after the quantile transformation of individual samples 
$(X_{1,j},\ldots,X_{n,j})$, 
the lower bound assumption on the joint density is much harder to ascertain. 
Consequently, it is unclear from the existing literature the extent of the validity of the largely 
theoretical assumption \eqref{eq: requirement low} beyond the restrictive condition on the 
lower bound of the joint density. 

Our results provide the validity of \eqref{eq: requirement low} in a broad collection of new scenarios as follows. 
When $X_{1:p} = \left(T_1(Z_{1}),\cdots,T_p(Z_{p})\right)$ follows a hidden pairwise Gaussian distribution, 
Corollary \ref{cor: nonlinear correlation hidden} {provides} 
\eqref{eq: requirement low} with $\kappa_2=\lambda_{\min}(\Sigma^z)$, 
Thus, the population predictive invertibility condition 
is satisfied 
as long as the covariance matrix $\Sigma^z$ is well conditioned.}

%
%

\subsection{Implications to high-dimensional additive models}

The prediction performance of additive models has also been carefully investigated in the 
high-dimensional setting through regularized estimation. 
{In the high-dimensional setting where $d^*>n$, e.g. $p>n$, 
the sample invertibility condition \eqref{eq: sample invert} would not hold for any $\kappa_0>0$ 
as the rank of the matrix $((B_j^{\intercal}(X_{i,j}),j\le p)^{\intercal}, i\le n)$ cannot be greater than $n$. 
A popular remedy to this impasse is to impose the sparsity condition that 
only a small unknown subset of components $f_1,\cdots, f_{p}$ are actually non-zero. 
This is referred to as the sparse additive model and has the natural interpretation that 
the response $Y$ depends on the design variables only through a small number of them. 
We use $s$, the number of non-zero $f_j$, 
and  the smoothness index of the nonzero $f_j$ to measure the complexity of the 
sparse additive model. } 
A core assumption {in} the theory {of penalized estimation in the} sparse additive model 
is the restricted eigenvalue and compatibility conditions.  
Let $\calI = \{j: f_j \neq 0\}$ be the unknown index set of real signals 
and $\kappa_0$ and $\xi_0$ be positive constants, the 
theoretical restricted eigenvalue and compatibility conditions can be defined as 
{\begin{equation}
\phi^{*}=\inf\left\{\frac{|\calI|^{2-q}\left\|\sum_{j=1}^{p}f_{j}(X_{j})\right\|^2_{L_2^{(0)}(P)}}
{\Big(\sum_{j\in \mathcal{J}}\big\|f_{j}(X_{j})\big\|_{L_2^{(0)}(P)}^q\Big)^{2/q}}: 
\frac{\sum_{j\in \mathcal{I}}\pen_j(f_j)}
{\sum_{j\in \mathcal{I}^{c}}\pen_j(f_j)}>\xi_0
\right\} \ge \kappa_0 
\label{eq: theory comp}
\end{equation}
with the convention $0/0=0$, where $q=2$ and $\calJ = \{1,\ldots,p\}$ for 
the restricted eigenvalue condition, $q=1$ and $\calJ = \calI$ for the compatibility coefficient, 
and $\pen_j(f_j)$, typically a certain norm of $f_j$ as a regularizer, is the penalty function. 

In high-dimensional linear regression, the sample restricted eigenvalue and compatibility 
conditions were respectively proposed in \citet{bickel2009simultaneous} and \citet{van2009conditions}. 
Condition \eqref{eq: theory comp}, which generalizes the population predictive invertibility condition 
\eqref{eq: requirement low} imposed in the low-dimensional setting, 
is comparable to the key invertibility conditions imposed in 
\cite{koltchinskii2010sparsity} and \cite{suzuki2013fast} for $q=2$ and 
\cite{meier2009high} and \cite{tan2017penalized} for $q=1$. } 

To make the dependence on the compatibility condition more explicit, {Theorem 1 of \cite{meier2009high} establishes that, in} the case where all the unknown functions are twice differentiable, the rate of convergence in terms of {the} in-sample prediction accuracy is 
$s (\log p/n)^{4/5}/\phi_n$, 
{where} $\phi_n$ is a sample version of $\phi^*$ defined in \eqref{eq: theory comp}. 
{Moreover, Theorem~2 of \cite{meier2009high} and its proof show that} 
as long as the theoretical compatibility condition \eqref{eq: theory comp} {holds,} 
$\phi_n$ and $\phi^{*}$ are of the same order and {the rate of the population prediction error 
is $s (\log p/n)^{4/5}/{\phi^{*}}$.}   
This directly illustrates the role of $\phi^{*}$ defined in the theoretical compatibility condition \eqref{eq: theory comp} on the rate of convergence. 


Despite the importance of \eqref{eq: theory comp} in theoretical justification of 
{regularized prediction in sparse} 
additive models, it has been typically imposed as a condition but without 
{further verification of} its validity  
other than {in some very special cases such as} 
the class of densities {of $X_{1:p}$} on $[0,1]^p$ uniformly bounded away 
from 0 and $\infty$. The result of the current paper on the minimum eigenvalue of the nonlinear correlation matrix sheds light on the {theoretical restricted eigenvalue and 
compatibility conditions} for additive models 
in the sense that {condition \eqref{eq: theory comp}} is satisfied with $\kappa_0$ being the 
minimum eigenvalue of the correlation matrix of {the latent pairwise Gaussian vector $Z_{1:p}$ 
as in Corollary \ref{cor: nonlinear correlation hidden}.}

\begin{corollary}
Suppose $\left(X_{1},X_{2},\cdots, X_{p}\right)$ follows a hidden Gaussian distribution 
with $X_j =T_j(Z_j)$ for a pairwise Gaussian 
vector $(Z_1,\ldots,Z_p)$ with $\Corr(Z_1,\ldots,Z_p) = \Sigma^z$ and 
some deterministic functions $T_j$ with $0<\Var(T_j(Z_j))<\infty$. 
Then, {condition} \eqref{eq: theory comp} holds with 
$\kappa_0=\lambda_{\min}(\Sigma^{z})$. 
In particular, {the theoretical restricted {eigenvalue} and compatibility conditions hold when 
$\lambda_{\min}(\Sigma^{z})$ is strictly bounded away from zero.} 
\label{cor: comp}
\end{corollary}

The above corollary implies that {condition \eqref{eq: theory comp}} holds for the Gaussian copula {model.} To the best of the authors' knowledge, this is a new connection of the theoretical restricted eigenvalue and compatibility conditions to {the widely-used model of multivariate dependency.} 

In addition to verifying the important condition \eqref{eq: theory comp}, 
{our results also provide the following connection between} 
the rate of convergence {in the estimation of} the individual components $f_j$ 
{and the prediction rate} 
\citep{meier2009high,koltchinskii2010sparsity,raskutti2012minimax, 
suzuki2013fast,tan2017penalized}. 

\begin{corollary}
Under the same assumption as Corollary \ref{cor: comp}, 
\begin{equation*}
{\lambda_{\min}(\Sigma^{z})\sum_{i=1}^{p} \|\fhat_i-f_i\|_{L_2^{(0)}(\P)}^2 
\leq \bigg\| \sum_{i=1}^{p} \fhat_i - \sum_{i=1}^{p} f_i \bigg\|_{L_2^{(0)}(\P)}^2
\le \lambda_{\max}(\Sigma^{z})\sum_{i=1}^{p} \|\fhat_i-f_i\|_{L_2^{(0)}(\P)}^2.}
\end{equation*}
\label{cor: covergence}
\end{corollary}

\section{
{Symmetric} Functions of iid Random Variables}
\label{sec: symiid}
In this section, we move beyond the pairwise Gaussianality and consider the extreme nonlinear correlation for symmetric functions of iid random variables. We first consider multiple  {nested} sums of iid random variables {to directly generalize} 
the results for a pair of {nested} sums established in Dembo, Kagan and Shepp (2001)
{and \cite{bryc2005maximum}.}  
In Section \ref{sec: symmetric}, we consider {the} class of symmetric functions defined on 
groups of iid random variables
and establish the extreme nonlinear correlation {in the much broader setting}. 

\subsection{{Nested} sums}
\label{sec: partial sum}
In this section, we consider the extreme nonlinear correlation for multiple {nested} sums of iid random variables. Specifically, given positive integers $m_1< m_2< \cdots <m_{p}$ 
and iid non-degenerate random variables $Y_1,Y_2,\ldots$, we consider 
\begin{equation}
X_j = S_{m_j} = \sum_{i=1}^{m_j} Y_i\quad \text{for}\;\; j=1,\ldots, p.
\label{eq: partial sum}
\end{equation}
Here, {the non-degeneracy} 
means that the distribution of the random variable is not concentrated at a {single} point. 
In the case of $p=2$, Dembo, Kagan and Shepp (2001) 
proved that the maximum correlation of $S_{m_1}$ and $S_{m_2}$ is equal to 
{$\sqrt{m_1/m_2}$} if $Y$ has finite second moment, 
and \cite{bryc2005maximum} proved the same result even 
without assuming the finite second order moment by investigating the characteristic functions 
of sums of $Y_i$. The following theorem extends their results from $p=2$ to general finite $p$. 
Further extensions to general symmetric functions of arbitrary groups of $Y_i$ are 
given in the next subsection. 

\begin{theorem}
Let $Y, Y_1, Y_2, \ldots$ be iid {non-degenerate random variables} 
and {$X_{1},X_{2},\cdots, X_{p}$} be nested sums of $Y_i$ 
{with sample sizes $1\le m_1\le\cdots\le m_p$} as defined 
in \eqref{eq: partial sum}. Then, 
\bel{eq: result 2} 
{\rho_{\max}^{NL}(X_{1:p};\nu,W) = \lambda_{\max}(R\circ W), \quad 
\rho_{\min}^{NL}(X_{1:p};\nu,W) = \lambda_{\min}(R\circ W),}
\eel
where $R = (R_{j,k})_{p\times p}$ is the matrix with elements 
{$R_{jk}=(m_j\wedge m_k)/\sqrt{m_j m_k}$,  
$\nu$ is taken as the counting measure for the extreme nonlinear correlations defined in   
 \eqref{NL-corr-max-stoch} and \eqref{NL-corr-min-stoch}, and $\circ$ denotes the Schur product.} 
If $Y$ has a finite second moment, then 
$R\in \R^{p\times p}$ is the correlation matrix of the {nested} sums 
$X_j=S_{m_j}$, $1\le j\le p$, so that \eqref{theme} holds with 
{$X_{1:p}$ for all measures $\nu$ 
and weights $W_{j,k}=W_{k,j}\ge 0$,}
$$
{\rho_{\max}^{NL}(X_{1:p};\nu,W) =\rho_{\max}^{L}(X_{1:p};\nu,W),\quad  
\rho_{\min}^{NL}(X_{1:p};\nu,W) =\rho_{\min}^{L}(X_{1:p};\nu,W).}
$$
\label{thm: partial sum}
\end{theorem}

{As discussed below Corollary \ref{cor-1}, for $m_j=j$ and large $p$, 
$\lambda_{\max}(R\circ W)/p$ with weight matrix $W_{j,k}$ 
is approximately the maximum eigenvalue of 
the operator $K_{W'}$ in $L_2([0,1])$ when $W_{j,k}=W'_{j/p,k/p}$ for a 
$W'_{s,t}$ continuous in $(s,t)\in [0,1]^2$.}  




{
\begin{proof}
As $f_j(X_j) = f_j(S_{m_j})$, $m_1\le\cdots\le m_p$, 
are symmetric functions of nested variable groups 
$\{Y_i, i\in G_j\}$ with $G_j=\{1,2,\cdots,m_j\}$ and $\cap_{j=1}^p G_j = G_1\neq\emptyset$, 
it follows from Theorem \ref{thm: symmetric} in the next subsection that 
{
\bes
\rho_{\max}^{NL}(X_{1:p},\nu,W) \le \lambda_{\max}(R\circ W),\quad 
\rho_{\min}^{NL}(X_{1:p},\nu,W) \geq \lambda_{\min}(R\circ W). 
\ees
We note that $\nu$ is the counting measure here. } 
It remains to prove that $\lambda_{\max}(R{\circ W})$ and $\lambda_{\min}(R{\circ W})$ 
are attainable by functions $f_j(X_j)$. This would be simple under the second moment condition on $Y$ as 
we may simply set $f_j(X_j) = X_j$ {to achieve $\Corr(X_{1:p})=R$.} 
In the case of $\E[Y^2]=\infty$, 
we prove that $R$ is in the closure of the {correlation} matrices generated by $(f_j(X_j), j\le p)$. 
This will be done below by proving 
\bel{pf-th-2-0}
\lim_{t\to 0+} \rho\big(\sin(t X_j - m_jc_t),\sin(t X_j - m_kc_t)\big) = R_{j,k},\quad 1\le j<k\le p, 
\eel
where $c_t\in (-\pi/2,\pi/2)$ is the solution of 
\bes
\E[ \sin(tY-c_t)] = 0,\quad \hbox{ or equivalently }\ \  \frac{\E[ \sin(tY)]}{\E[\cos(tY)]} = \tan(c_t). 
\ees
{We shall choose the sequence $t\to 0+$ such that for each $t$, $\P\{\sin(t(Y_1-Y_2))=0\}<1$ so that $\P\{\sin(tY)=0\}<1$ 
and $\P\{\sin(tY-c_t)=0\}<1$. This is always feasible when $Y$ is non-degenerate.} 

As $\E[ \sin(tY)] \to 0$ and $\E[\cos(tY)]\to 1$, 
it suffices to consider small $t>0$ satisfying $|c_t|\le 1$. 
Let $Y' = tY - c_t$. 
As $\big|\sin(y)(1-\cos(y))\big| \le \sin^2(y) + 2|\sin(y)| I_{\{ |y| > 2\}}$, we have 
\bel{pf-th-2-1}
\Big| \E\big[\sin(Y')\cos(Y')\big] \Big|
&=& \Big| \E\big[\sin(Y')(1-\cos(Y'))\big] \Big|
\cr &\le& \E[ \sin^2(Y')\big]  + \sqrt{\E[ \sin^2(Y')\big] \P\{|Y| > 1/t\}}. 
\eel
Let $Y_i' = tY_i - c_t$ and $S'_{a:m} = \sum_{i=a}^m Y'_i$. 
We shall prove that for $a\le b\le m\le n$
\bel{pf-th-2-2}
\lim_{t\to 0+} \rho\big(\sin(S'_{a:m}),\sin(S'_{b:n})\big) = \frac{(m-b+1)}{(m-a+1)^{1/2}(n-b+1)^{1/2}}. 
\eel
This implies \eqref{pf-th-2-0} with $a=b=1$, $m=m_j$ and $n=m_k$, but the more general 
$a$ and $b$ would provide {the} extension to sums of arbitrary subgroups of $Y_i$ later 
in Corollary \ref{cor-6}. 

Let $f_{a,m} = \sin(S'_{a:m})$. As $\sin(y+z) = \sin(y)\cos(z)+\cos(y)\sin(z)$. We write 
\bes
f_{a,m}  = \sum_{u=a}^m f_{a,m,u}\ \ \hbox{ where }\ \ 
f_{a,m,u} = \bigg(\prod_{i=a}^{u-1}\cos(Y'_i)\bigg)\sin(Y'_u)\cos(S'_{(u+1):m}). 
\ees
Let $a\le b\le m\le n$. 
As $\E[ \sin(Y'_a)]=0$, we have $\E[f_{a,m}]=0$ and $\E[f_{a,m,u}f_{b,n,v}]=0$ 
for $a\le u < b$ or for $m < v \le n$. 
For $b \le u\wedge v\le u\vee v \le m$, 
{
\bes
&& f_{a,m,u}f_{b,n,v} 
\\ \nonumber &=& \bigg(\prod_{i=a}^{u-1}\cos(Y'_i)\bigg)\sin(Y'_u)\cos(S'_{(u+1):m}) 
\bigg(\prod_{i=b}^{v-1}\cos(Y'_i)\bigg)\sin(Y'_v)\cos(S'_{(v+1):n})
\\ \nonumber &=& \sin(Y'_{u\wedge v})\cos(Y'_{u\wedge v})\sin(Y'_{u\vee v})g(Y'_i, a\le i\le n, i\neq u\wedge v)
\ees
}
for a certain function $g$ bounded by 1. Thus, as a consequence of \eqref{pf-th-2-1}
\bes
\big| \E\big[f_{a,m,u}f_{b,n,v}\big]\big|
&\le & \Big| \E\big[\sin(Y')\cos(Y')\big] \Big| \E\big[ |\sin(Y')|\big]
\cr &\le & \E\big[\sin^2(Y')\big] \Big(\sqrt{ \E\big[\sin^2(Y')\big]} +\sqrt{\P\big\{|Y| > 1/t \big\}}\Big)
\ees
for $b \le u\wedge v < u\vee v \le m$. Moreover, for $b \le u \le m$, 
\bes
&& \E\big[f_{a,m,u}f_{b,n,u}\big]
\cr & = & \E\big[\sin^2(Y'_u)\big] \E\bigg[ \bigg(\prod_{i=a}^{b-1}\cos(Y'_i)\bigg)
\bigg(\prod_{i=b}^{u-1}\cos^2(Y'_i)\bigg)\cos\big(S'_{(u+1):m}\big)\cos\big(S'_{(u+1):n}\big)\bigg]. 
\ees
Thus, as $Y'_i = tY_i - c_t\to 0$ in probability, we find that for all $a\le b\le m\le n$  
\bes
\lim_{t\to 0+}\frac{\E\big[\sin(S'_{a:m}) \sin(S'_{b:n})\big]}{\E\big[\sin^2(Y')\big]} 
= \lim_{t\to 0+}\sum_{u=a}^m\sum_{v=b}^n \frac{\E\big[f_{a,m,u}f_{b,n,v}\big]}{\E\big[\sin^2(Y')\big]} 
= \#\big\{b \le u = v \le m\big\}. 
\ees
This implies \eqref{pf-th-2-2} and completes the proof. 
\end{proof}
}


\subsection{Symmetric functions of groups of variables}
\label{sec: symmetric}

In this section, we consider a broader setting than {nested} sums considered in Section \ref{sec: partial sum}. We use $\{Y_{i}\}_{i\geq 1}$ to denote an infinite sequence of iid random variables and define 
random vectors $\bX_j = (Y_i, i\in G_j)$ for arbitrary sets of positive integers $G_j$ 
of finite size $m_j = |G_j|<\infty$. 
Again we are interested in the extreme nonlinear correlation 
{among $\bX_1,\ldots,\bX_p$.} 

As $\bX_j$ are vectors, 
we adjust the definition of the extreme nonlinear 
{correlations in \eqref{NL-corr-max-stoch} and \eqref{NL-corr-min-stoch} as follows: 
Given a $p\times p$ symmetric matrix $W=(W_{j,k})$ with $W_{j,k}\ge 0$, define} 
\bel{NL-corr-max-group}
\rho^{NL}_{\max,\,\rm symm} 
= {\rho^{NL}_{\max,\,\rm symm}(\bX_{1}, \cdots, \bX_{p},W)  
= \sup_{f_{1:p}\in\calF_{1:p}} \lam_{\max}\big(K_{W,f_{1:p}}\big),}
\eel
where $\calF_{1:p}= \{(f_1,\ldots,f_p): 0 < \Var(f_j(\bX_j))<\infty, f_j(y_1,\ldots,y_{m_j})$ symmetric$\ \forall 1\leq j\leq p\}$ and $K_{W,f_{1:p}}=\big(\Corr\big(f_j(\bX_j),f_k(\bX_k)\big)W_{j,k}\big)_{p\times p}$. 
Correspondingly, {define  
\bel{NL-corr-min-group}
\rho^{NL}_{\min,\,\rm symm} 
= \rho^{NL}_{\min,\,\rm symm}(\bX_{1}, \cdots, \bX_{p},W)  
= \inf_{f_{1:p}\in\calF_{1:p}} \lam_{\min}\big(K_{W,f_{1:p}}\big).
\eel
We omit $\nu$ in the notation as it is taken as the counting measure in $\{1,\ldots,p\}$ 
without loss of generality as discussed below \eqref{extreme-NL-kernel}.  
Here, the symmetry of $f_j$ means permutation invariance, 
$f_j(y_1,\ldots,y_{m_j}) = f_j(y_{i_1},\ldots,y_{i_{m_j}})$ 
for all permutations $i_1,\ldots,i_{m_j}$ of $1,\ldots,m_j$.}  
To avoid confusion, {we call} the above quantities 
extreme symmetric nonlinear correlations.  
We extend Theorem~\ref{thm: partial sum} to groups satisfying the following assumption.  

\noindent {\bf Assumption C:} {\it There exist certain sets $G_{0,j}$ of positive integers such that}
\bes
|G_{0,j}\cap G_{0,k}| = {\big(|G_{j}\cap G_{k}| -1\big)_+}\ \forall 1\le j < k\le p,\quad 
|G_{0,j}|\le |G_j|-1\ \forall 1\le j\le p. 
\ees
Assumption C holds when $\cap_{j=1}^p G_j \neq\emptyset$, as we can simply 
set $G_{0,j} = G_j\setminus \{i_0\}$ for a fixed $i_0\in \cap_{j=1}^p G_j$. Hence, for the special case that $G_j$ are nested with
$\emptyset \neq G_1 \subset G_2\subset \cdots \subset G_{p}$, Assumption C holds automatically.  
{However, $G_{0,j}$ do not need to have anything to do with $G_j$ beyond the specified conditions 
on their size and the size of their intersections.} 


\begin{theorem}\label{thm: symmetric}
Let $Y, Y_1, Y_2,\ldots$ be iid non-degenerate random variables and 
$\bX_j = (Y_i, i\in G_j)$ for arbitrary groups of positive integers 
$G_1,\ldots,G_p$ of finite size $m_j = |G_j|<\infty$. 
Let $\rho^{NL}_{\max,\,\rm symm}$ and $\rho^{NL}_{\min,\,\rm symm}$ 
be the extreme symmetric nonlinear correlations among {$\bX_1,\ldots,\bX_p$} as defined in 
\eqref{NL-corr-max-group} and \eqref{NL-corr-min-group} 
{with weight matrix $W$.}  
Let $R^{(\ell)}\in \R^{p\times p}$ be the matrix with elements 
\begin{equation}
R^{(\ell)}_{j,k} 
= {|G_j\cap G_k|\choose \ell}  {|G_j|\choose \ell}^{-1/2} {|G_k|\choose \ell}^{-1/2}
\label{eq: general R}
\end{equation}
for $1\le\ell\le\ell^*$, {with the convention $0/0=0$,} 
where 
$\ell^*=\max_{1\le j\le p}|G_j|$. Then, 
\bel{new-th-3-1} 
\rho_{\max, {\rm symm}}^{NL}
=\lambda_{\max}(R{\circ W}), 
\ \ \rho^{NL}_{\min,\,\rm symm} 
= \min_{1\le\ell\le\ell^*} \lam_{\min}\big({\big(R^{(\ell)}\circ W\big)_{J^{(\ell)},J^{(\ell)}}}\big), 
\eel
{with $R=R^{(1)}$, $J^{(\ell)}=\{1\le j\le p: |G_j|\ge \ell\}$ and 
$\circ$ being the Schur product.} 
If in addition Assumption C holds, then 
\bel{new-th-3-2} 
\rho^{NL}_{\min,\,\rm symm} = \lam_{\min}\big(R{\circ W}\big). 
\eel
\end{theorem}

{The first part of \eqref{new-th-3-1} asserts that the maximum symmetric nonlinear correlation 
is identical to its linear version, while the second part gives a formula for the minimum 
symmetric nonlinear correlation. Under Assumption C, \eqref{new-th-3-2} asserts the 
equality between the minimum symmetric nonlinear correlation and its linear version.} 
The connection between Theorem \ref{thm: partial sum} and 
Theorem \ref{thm: symmetric} can be built under the observation 
that $f_j(\sum_{i=1}^{m_j} Y_i)$ is a symmetric function 
of $\bX_j=\{Y_i\}_{i\in G_j}$ when $G_j=\{1,2,\cdots,m_j\}$, 
and the corresponding index sets $G_j$ 
satisfy the Assumption C due to the nested structure of $\{G_j\}_{1\leq j\leq p}$. 
For the case $p=2$, 
Theorem \ref{thm: symmetric} serves as an extension of Dembo, Kagan and Shepp (2001) 
from  {functions $f_j(\sum_{i=1}^{m_j} Y_i)$ of the two sums to any symmetric functions of iid random variables and of \cite{yu2008maximal} from two $f_j(\sum_{i\in G_j} Y_i)$ with arbitrary $G_j$.}

An interesting aspect of Theorem \ref{thm: symmetric}
is that {under assumption C} 
the extreme symmetric nonlinear correlation is attained by {sums of the form} 
\begin{equation}
f_j(\bX_j)= \sum_{i\in G_j} h_0(Y_i) \quad \text{for}\; 1\leq j\leq p, 
\label{eq: achievable function}
\end{equation}
for any function $h_0$ with $0<\Var(h_0(Y))<\infty$, e.g. $h_0(Y_i)=Y_i$ when $Y_i$ has finite variance. 
That is to say, among symmetric functions, the most extreme multivariate correlations are achieved 
by the linear summation of iid random variables. 
{The following corollary, based on Theorem \ref{thm: symmetric} and \eqref{pf-th-2-2} in the proof 
of Theorem \ref{thm: partial sum}, asserts that the extreme symmetric nonlinear correlations 
for groups of $Y_i$ are achieved 
by functions of the corresponding sums of $Y_i$ without assuming the finite second moment condition. 

\begin{corollary}\label{cor-6} Let $\bX_j = (Y_i, i\in G_j)$ and $S_{G_j} = \sum_{i\in G_j}Y_i$ 
with iid non-degenerate $Y_i$. Then, 
\bes
& \rho^{NL}_{\max,\,\rm symm}({\bX_{1}, \cdots, \bX_{p},W}) 
= \rho^{NL}_{\max}\big({(S_{G_1},\ldots,S_{G_p}),\nu,W}\big) 
= \lam_{\max}\big(R{\circ W}\big), 
\cr & \rho^{NL}_{\min,\,\rm symm}({\bX_{1}, \cdots, \bX_{p},W}) 
= \rho^{NL}_{\min}\big({(S_{G_1},\ldots,S_{G_p}),\nu,W}\big) 
= \lam_{\min}\big(R{\circ W}\big),
\ees
under Assumption C, {where $\nu$ is taken as the counting measure in 
the extreme nonlinear correlations 
in \eqref{NL-corr-max-stoch} and \eqref{NL-corr-min-stoch}.} 
Consequently, \eqref{theme} holds for $X_j=S_{G_j}$ when $\E[Y^2]<\infty$. 
\end{corollary} 
}

The proof of Theorem \ref{thm: symmetric}
relies on the \cite{hoeffding1948class, hoeffding1961strong} decomposition of 
{symmetric functions} of random variables, stated as Lemma \ref{lm-3} below; 
See Lemma 1 in \cite{hoeffding1961strong}, the decomposition lemma in \cite{efron1981jackknife}, 
and Lemma 1 in \cite{dembo2001remarks}. 

 
\begin{lemma}\label{lm-3} 
Let $\bY = (Y_1,\cdots,Y_m)$ with iid components $Y_i$ and $f_0(\bY)=f_0(Y_1,\cdots,Y_m)$ 
with a symmetric function $f_0(y_1,\ldots,y_m)$. 
Suppose $\E[f_0(\bY) ]=0$ and $\E[f_0^2(\bY)]<\infty$. 
Define $f_{0,1}(y_1)=\E[f_0(\bY)|Y_1=y_1]$ and for $k=2,\ldots,m$ define 
\bes
f_{0,k}({y_{1:k}}) = \E\left[ f_0(\bY) 
- \sum_{j=1}^{k-1} \sum_{1\le i_1<\cdots<i_j\le m}f_{0,j}(Y_{i_1},\ldots,Y_{i_j})\Bigg|
{Y_{1:k}=y_{1:k}}\right]. 
\ees
Then, the following expansion holds, 
\begin{equation}
f_0(\bY)
=\sum_{\ell=1}^m \sum_{1\le i_1<\cdots<i_\ell\le m}f_{0,\ell}(Y_{i_1},\ldots,Y_{i_\ell}), 
\label{eq: decomposition}
\end{equation}
and that for all $s =1,\cdots, \ell$ and $\ell =1,\cdots, m$
\begin{equation}
\E\Big[ f_{0,\ell}(Y_{i_1},\ldots,Y_{i_\ell})\Big| \{Y_{i_1},\ldots,Y_{i_\ell}\}\backslash Y_{i_s}\Big] =0.
\label{eq: orthog}
\end{equation}
Consequently, 
\bel{SS}
\E\big[f_0^2(\bY)\big] 
=\sum_{\ell=1}^m {m\choose \ell} \E\Big[f_{0,\ell}^2({Y_{1:\ell}})\Big]. 
\eel
\end{lemma}


\begin{proof}[Proof of Theorem \ref{thm: symmetric}]  
{Assume without generality $\E[f_j(\bX_j)]=0, \E[f_j^2(\bX_j)]=1$ for all $j$ as 
$\rho_{\max, {\rm symm}}^{NL}$ and $\rho^{NL}_{\min,\,\rm symm}$ 
are defined through the correlations between $f_j(\bX_j)$ and $f_k(\bX_k)$.} 
Let $G^{(\ell)} =\{ (i_1, \cdots, i_\ell): i_1<\cdots<i_\ell,    
i_{s} \in G \; \text{for} \; 1\leq s\leq \ell\}$ for all subsets $G$ of positive integers. 
Since $f_j(\bX_j)$ are symmetric functions of {$\{Y_i\}_{i\in G_j}$,} 
\eqref{eq: decomposition} gives  
\begin{equation}
f_j(\bX_j)= \sum_{\ell=1}^{m_j} \sum_{{(i_1, \cdots, i_\ell) \in G_{j}^{(\ell)}}} 
f_{j,\ell}(Y_{i_1},\ldots,Y_{i_\ell}). 
\label{eq: expression of symmetric}
\end{equation}
We first apply \eqref{eq: orthog} and obtain the following expression for the cross-product,
\bes
\E\Big[f_{j,\ell}(Y_{i_1},\ldots,Y_{i_\ell}) f_{k,{\ell'}}(Y_{i'_1},\ldots,Y_{i'_{\ell'}}) \Big] = 0 
\ees
when $\{i_1,\ldots,i_\ell\} \neq \{i'_1,\ldots,i'_{\ell'}\}$. It follows that 
\bel{pf-th-3-1}
\E f_j(\bX_j) f_k(\bX_k) 
&= &\E \sum_{\ell=1}^{|G_{j}\cap G_{k}|} \sum_{{(i_1,\ldots,i_\ell) \in (G_{j}\cap G_{k})^{(\ell)}}}
f_{j,\ell}(Y_{i_1},\ldots,Y_{i_\ell})
f_{k,\ell}(Y_{i_1},\ldots,Y_{i_\ell})
\cr &= & \sum_{\ell=1}^{\ell^*} {|G_{j}\cap G_{k}|\choose \ell}\E\Big[ f_{j,\ell}({Y_{1:\ell}})
f_{k,\ell}({Y_{1:\ell}})\Big]
\eel
{with} the convention ${m\choose \ell}=0$ for $\ell>m$. 
Let $R^{(\ell)}\in \R^{p\times p}$ be the matrix defined in \eqref{eq: general R}. 
Let $g_{j,\ell} = g_{j,\ell}({Y_{1:\ell}})={m_j\choose \ell}^{1/2}f_{j,\ell}({Y_{1:\ell}})$.  
{For $u=(u_1,\ldots,u_p)^{\intercal}$ with $\|u\|_2=1$, \eqref{pf-th-3-1} provides
\bes
u^{\intercal}K_{W,f_{1:p}} u
& = & \E \left(\sum_{j=1}^p\sum_{k=1}^p W_{j,k}u_ju_k f_{j}(\bX_j)f_{k}(\bX_k)\right)
\cr &=&\sum_{j=1}^p\sum_{k=1}^p W_{j,k}u_ju_k
 \sum_{\ell=1}^{\ell^*} {|G_{j}\cap G_{k}|\choose \ell}\E\left[ f_{j,\ell}(Y_{1:\ell})
f_{k,\ell}(Y_{1:\ell})\right]  
\cr &=& \sum_{\ell=1}^{\ell^{*}} \E\left[\sum_{j\in J^{(\ell)}} \sum_{k\in J^{(\ell)}}  
R^{(\ell)}_{j,k}W_{j,k}u_ju_k g_{j,\ell}(Y_{1:\ell})g_{k,\ell}(Y_{1:\ell})\right]
\cr &\le & \max_{1\le\ell\le \ell^{*}} \lam_{\max}\Big(\big(R^{(\ell)}\circ W\big)_{J^{(\ell)},J^{(\ell)}}\Big)  
\sum_{\ell=1}^{\ell^{*}} \E \left[\sum_{j\in J^{(\ell)}} u_j^2g_{j,\ell}^2(Y_{1:\ell})\right] 
\cr &=& \max_{1\le\ell\le \ell^{*}}\lam_{\max}
\Big(\big(R^{(\ell)}\circ W\big)_{J^{(\ell)},J^{(\ell)}}\Big) 
\sum_{j=1}^p\E \left[u_j^2f_j^2(\bX_j)\right]
\cr &=& \max_{1\le\ell\le \ell^{*}}\lam_{\max}
\Big(\big(R^{(\ell)}\circ W\big)_{J^{(\ell)},J^{(\ell)}}\Big),
\ees
where the second to the last equality follows from \eqref{SS} and 
the fact that $g_{j,\ell} = 0$ for $\ell> m_j = |G_j|$.
Similarly, for all $u=(u_1,\ldots,u_p)^{\intercal}$ with $\|u\|_2=1$, 
\bes
u^{\intercal}K_{W,f_{1:p}} u
\ge \min_{1\le\ell\le \ell^{*}}\lam_{\min}\Big(\big(R^{(\ell)}\circ W\big)_{J^{(\ell)},J^{(\ell)}}\Big).
\ees
Thus, by \eqref{NL-corr-max-group} and \eqref{NL-corr-min-group}, 
\bel{eq: general bound} 
&& \rho_{\max,{\rm symm}}^{NL}  \le \max_{1\le\ell\le \ell^*}
\lam_{\max}\Big(\big(R^{(\ell)}\circ W\big)_{J^{(\ell)},J^{(\ell)}}\Big), 
\cr && \rho_{\min,{\rm symm}}^{NL} \ge \min_{1\le\ell\le \ell^{*}}
\lam_{\min}\Big(\big(R^{(\ell)}\circ W\big)_{J^{(\ell)},J^{(\ell)}}\Big).
\eel
To prove \eqref{eq: general bound} holds with equality, 
we pick a specific $f_{1:p}$ for each $\ell$ as follows. 
Let $h_0$ be a function satisfying $\E[h_0(Y)]=0$ and $\E[h_0^2(Y)]=1$. 
For $j\in J^{(\ell)}$ define 
\bel{h_0}
h^{(\ell)}_{0,j}(\bX_j) = 
{|G_j|\choose \ell}^{-1/2}\sum_{|S|=\ell,S\subseteq G_j} \prod_{i\in S}h_0(Y_i) 
\eel
as symmetric functions of $\bX_j$. 
For $\{j,k\}\subset  J^{(\ell)}$ we have 
\bes
\E\Big[ h^{(\ell)}_{0,j}(\bX_j)h^{(\ell)}_{0,k}(\bX_k)\Big] 
= {|G_j\cap G_k|\choose \ell}  {|G_j|\choose \ell}^{-1/2} {|G_k|\choose \ell}^{-1/2} 
I_{\{\ell\le |G_j\cap G_k|\}}
= R^{(\ell)}_{j,k}.
\ees
Thus, when $f_j(\bX_j)=h^{(\ell)}_{0,j}(\bX_j)$ for $j \in J^{(\ell)}$, we have 
$(K_{W,f_{1:p}})_{J^{(\ell)},J^{(\ell)}}=\big(R^{(\ell)}\circ W\big)_{J^{(\ell)},J^{(\ell)}}$. 
As this holds for every $\ell\le\ell^*$, \eqref{eq: general bound} holds with equality 
by \eqref{NL-corr-max-group} and \eqref{NL-corr-min-group}. 
We note that $\E[f_j(\bX_j)f_k(\bX_k)]\neq 0 = R^{(\ell)}_{j,k}$ 
typically holds for $j\not\in J^{(\ell)}$ or $k\not\in J^{(\ell)}$ 
as $\E[f_j^2(\bX_j)]=1$. 

It remains to prove that the extreme eigenvalues in \eqref{eq: general bound} 
are achieved with $\ell=1$. For the maximum eigenvalue, we notice that by \eqref{eq: general R}, 
\bes
R^{(\ell)}_{j,k} 
&=& {|G_j\cap G_k|\choose \ell}  {|G_j|\choose \ell}^{-1/2} {|G_k|\choose \ell}^{-1/2}
I_{\{|G_j\cap G_k|\ge \ell\}}
\cr &=& 
\frac{|G_j\cap G_{k}|(|G_j\cap G_{k}|-1)\cdots (|G_j\cap G_{k}|-l+1)I_{\{|G_j\cap G_k|\ge \ell\}}}
{\sqrt{|G_{j}|(|G_{j}|-1)\cdots (|G_{j}|-l+1)}\sqrt{|G_{k}|(|G_{k}|-1)\cdots (|G_{k}|-l+1)}}
\cr &\leq& \frac{|G_j\cap G_{k}|}{\sqrt{|G_{j}| \cdot |G_{k}|}}, 
\ees
so that $0\leq R^{(\ell)}_{{j,k}}W_{j,k}\leq R^{(1)}_{{j,k}}W_{j,k}$ 
for all $\{j,k\}\subseteq J^{(\ell)}$. 
Thus, 
\bel{eq: upper geneal} 
\lam_{\max}\Big(\big(R^{(\ell)}\circ W\big)_{J^{(\ell)},J^{(\ell)}}\Big) 
\le \lam_{\max}\Big(\big(R^{(1)}\circ W\big)_{J^{(\ell)},J^{(\ell)}}\Big) 
\le \lam_{\max}\big(R^{(1)}\circ W\big). 
\eel 
due to the element-wise positiveness of $\big(R^{(\ell)}\circ W\big)_{J^{(\ell)},J^{(\ell)}}$. 
As \eqref{eq: general bound} holds with equality, this completes the proof of \eqref{new-th-3-1}.}  

The remaining of the proof is to characterize 
$\min_{1\le\ell\le\ell^*} \lam_{\min}\big({\big(R^{(\ell)}\circ W\big)_{J^{(\ell)},J^{(\ell)}}}\big)$ under Assumption C. 
As the result can be of independent interest, we state {it} 
in the following lemma and supply a proof immediately after the lemma.

\begin{lemma}
Under Assumption C, we have {
$$\min_{1\le\ell\le\ell^*} \lam_{\min}\big(\big(R^{(\ell)}\circ W\big)_{J^{(\ell)},J^{(\ell)}}\big)
=\lam_{\min}\big(R\circ W\big)
$$
where $R^{(\ell)}$ are defined in \eqref{eq: general R}, $R^{(1)}=R$ 
and $\ell^*=\max_{1\le j\le p}|G_j|$.} 
\label{lem: lower general}
\end{lemma}

\noindent {\it Proof of Lemma \ref{lem: lower general}.}
Under Assumption C, we set 
\bes
g^{({\ell-1})}_{0,j}(\bX_j) = {|G_{j}|-1\choose {\ell-1}}^{-1/2}\sum_{|S|={\ell-1},S\subseteq G_{0,j}} \prod_{i\in S}h_0(Y_i),\quad {j\in J^{(\ell)}, 2\le \ell\le \ell^*,}
\ees
with the $h_0$ in \eqref{h_0}. 
Similar to the proof of {the first part of} \eqref{eq: general bound} with equality, we have 
\bes
&& \E\Big[ g^{(\ell-1)}_{0,j}(\bX_j)g^{(\ell-1)}_{0,k}(\bX_k)\Big] 
\cr &=& {|G_{0,j}\cap G_{0,k}|\choose \ell-1}{|G_{j}|-1\choose \ell - 1}^{-1/2} {|G_{k}|-1\choose \ell - 1}^{-1/2} 
I_{\{|G_{0,j}\cap G_{0,k}|\ge \ell -1\}}
\cr &=& {{(|G_{j}\cap G_{k}| -1)_{+}} \choose \ell-1}{|G_{j}| -1 \choose \ell - 1}^{-1/2} 
{|G_{k}| -1\choose \ell - 1}^{-1/2} 
I_{\{ |G_{j}\cap G_{k}|\ge \ell\}}. 
\ees
{For $j=k$, $\Var\big( g^{(\ell-1)}_{0,j}(\bX_j)\big)\le 1$ as $|G_{0,j}|\le |G_j|-1$. 
It follows that, for} $|G_{j}\cap G_{k}|\geq 1,$
\bes
R_{j,k}\E\Big[ g^{(\ell-1)}_{0,j}(\bX_j)g^{(\ell-1)}_{0,k}(\bX_k)\Big] 
{\begin{cases} = R^{(\ell)}_{j,k}, & j\neq k \hbox{ in } J^{(\ell)}, 
\cr \le R^{(\ell)}_{j,k}, & j=k \in J^{(\ell)}. \end{cases}}
\ees
For the case $|G_{j}\cap G_{k}|= 0$, the above 
{relationship} trivially holds with both sides equal to zero. 
{It follows that when $\lam_{\min}(R\circ W)\le 0$, 
\bes
&& \lam_{\min}\big(\big(R^{(\ell)}\circ W\big)_{J^{(\ell)},J^{(\ell)}}\big)
\cr &=& \min_{\|\bu_{J^{(\ell)}}\|_2=1} 
\sum_{j\in J^{(\ell)}}\sum_{k\in J^{(\ell)}}u_ju_k R_{j,k}^{(\ell)} W_{j,k}
\cr &\ge & \min_{\|\bu\|_2=1} 
\sum_{j\in J^{(\ell)}}\sum_{k\in J^{(\ell)}}
u_ju_k R_{j,k}W_{j,k}
\E\Big[ g^{(\ell-1)}_{0,j}(\bX_j)g^{(\ell-1)}_{0,k}(\bX_k)\Big] 
\cr &\ge & \min_{\|\bu\|_2=1} \lam_{\min}(R\circ W) 
\E\bigg[\sum_{j\in J^{(\ell)}}\Big(u_jg^{(\ell-1)}_{0,j}(\bX_j)\Big)^2\bigg]
\cr &\ge & \lam_{\min}(R\circ W),
\ees
where the last inequality holds due to the fact that 
$$ \E\Big[ (g^{(\ell-1)}_{0,j}(\bX_j))^2\Big] =
 {{|G_{0,j}}|\choose \ell-1}{|G_{j}|-1\choose \ell - 1}^{-1} 
I_{\{|G_{0,j}|\ge \ell -1\}}\leq 1.
$$
In the general case, we notice that $\lam_{\min}(R\circ(W - cI_{p\times p}))\le 0$ 
and $W_{j,k}-cI_{\{j=k\}}\ge 0$ for all $1\le j,k\le p$ {and} $c=\min_{1\le j\le p} W_{j,j}$, so that 
\bes
\min_{f_{1:p}\in \calF_{1:p}}\lam_{\min}\big(K_{W,f_{1:p}}\big) - c
&=& \min_{f_{1:p}\in \calF_{1:p}}\lam_{\min}\big(K_{1,f_{1:p}}\circ(W-cI_{p\times p})\big)
\cr &=& \lam_{\min}\big(R\circ(W-cI_{p\times p})\big)
\cr &=& \lam_{\min}\big(R\circ W\big) - c 
\ees 
as the diagonal of $K_{1,f_{1:p}}$ and $R$ are both $I_{p\times p}$.} 
\end{proof}

{
\section*{Acknowledgments}
The authors would like to thank the Associate Editor and two referees whose constructive comments 
led to the examples in Subsection 2.3.}

\section{Appendix} 

We prove Lemmas \ref{lm-1} and \ref{lm-2} in this Appendix. 

\begin{proof}[Proof of Lemma \ref{lm-1}] 
Let $g_t(x) = h(t)f_t(x)\big/\big\{\E\big[f^2_t(X_t)\big]\big\}^{1/2}$ 
{with $h$ satisfying $\|h\|_{L_2(\nu)}^2=1$.} 
As $\int_{\calT}\E[g_t^2(X_t)]\nu(dt) = \|h\|_{L_2(\nu)}^2{=1}$, $f_{\calT}\in\calF_{\calT}$ 
implies $g_{\calT}\in\calF_{\calT}$, so that by \eqref{NL-corr-max-stoch} 
\bes
\rholam^{NL}_{\max} 
&=& \sup_{f_{\calT}\in\calF_{\calT}}\sup_{\|h\|_{L_2(\nu)}=1} 
\int_{s\in \calT}\int_{t\in \calT} \rho\left(f_s(X_s),f_t(X_t)\right){W_{s,t}}h(s)h(t)\nu(ds)\nu(dt)
\cr &\le & \sup_{g_{\calT}\in\calF_{\calT}} 
\frac{\int_{s\in \calT}\int_{t\in \calT} \E\big[g_s(X_s),g_t(X_t)\big]{W_{s,t}} \nu(ds)\nu(dt)}
{\int \E\big[g_t^2(X_t)\big] \nu(dt)}. 
\ees
On the other hand, letting 
$h(t) = \big\{\E\big[f^2_t(X_t)\big]\big/\int_{t\in \calT} \E\big[f_t^2(X_t)\big] \nu(dt)\big\}^{1/2}$, 
we have 
\bes
\rholam^{NL}_{\max} 
&\ge & 
\int_{s\in \calT}\int_{t\in \calT} \rho\left(f_s(X_s),f_t(X_t)\right){W_{s,t}}h(s)h(t)\nu(ds)\nu(dt)
\cr & = &  
\frac{\int_{s\in \calT}\int_{t\in \calT} \E\big[f_s(X_s),f_t(X_t)\big] {W_{s,t}}\nu(ds)\nu(dt)}
{\int_{t\in \calT} \E\big[f_t^2(X_t)\big] \nu(dt)}. 
\ees
for all $f_{\calT}\in \calF_{\calT}$. Thus, \eqref{NL-corr-max-stoch} and \eqref{eq: multi extreme} 
are equivalent. We omit the proof of the equivalence between \eqref{NL-corr-min-stoch}
and \eqref{eq: multi min extreme} as it can be established by the same argument. 
\end{proof}


\begin{proof}[Proof of Lemma \ref{lm-2}] 
Let $h$ be a function on $\calT$ with $\|h\|_{L_2(\nu)}=1$ 
{and $B_n\subset \calT$ be as in Assumption~A. 
Let $\{X_t^{(i)}, t\in\calT\}_{1\leq i\leq m-1}$ be iid copies of $X_{\calT}$. 
Because $\big|\E\big[X_sX_t\big]\big|\le \E\big[|X_s^{(i)}X_t^{(i)}|\big] \le 1$ and $|K_W(s,t)|\le W_{s,t}$, by Cauchy-Schwarz
\bes
&& \E \int_{B_n}\int_{B_n} \bigg(\prod_{i=1}^{m-1} \big|X_s^{(i)}X_t^{(i)}\big|\bigg) \Big|K_W(s,t) h(s)h(t)\Big| \nu(ds)\nu(dt)
\cr &\le& {\bigg(\int_{B_n}\int_{B_n}W_{s,t}^2\nu(ds)\nu(dt)\bigg)^{1/2}}
<\infty. 
\ees
Thus the exchange of expectation and integration is allowed in the following derivation: 
\bes
&& \int_{B_n}\int_{B_n} \big(\E[X_sX_t]\big)^{m-1}K_W(s,t) h(s)h(t)\nu(ds)\nu(dt) 
\cr & = & \E \int_{B_n}\int_{B_n} 
\bigg(\prod_{i=1}^{m-1} \big(X_s^{(i)}X_t^{(i)}\big)\bigg) K_W(s,t) h(s)h(t) \nu(ds)\nu(dt)
\cr &=& \E \int_{B_n}\int_{B_n} K_W(s,t) \bigg\{{h(s)}\prod_{i=1}^{m-1} X_s^{(i)} \bigg\} 
\bigg\{{h(t)}\prod_{i=1}^{m-1} X_t^{(i)} \bigg\} \nu(ds)\nu(dt)
\cr &\le & \rholam^{L}_{\max}\int \E\bigg[\bigg(I\{t\in B_n\}h(t)\prod_{i=1}^{m-1}X_t^{(i)}\bigg)^2\bigg] \nu(dt)
\cr & = & \rholam^{L}_{\max} \int_{B_n} h^2(t)\nu(dt). 
\ees
{Moreover,} as the exchange of expectation and integration is allowed, 
\bes
&& \int_{B_n}\int_{B_n} {\big(\E[X_sX_t]\big)^{m-1}K_W}(s,t) {h}(s){h}(t)\nu(ds)\nu(dt) 
\cr &=& \E \int_{B_n}\int_{B_n} K_W(s,t) \bigg\{{h(s)}\prod_{i=1}^{m-1} X_s^{(i)} \bigg\} 
\bigg\{{h(t)}\prod_{i=1}^{m-1} X_t^{(i)} \bigg\} \nu(ds)\nu(dt)
\cr & \ge & \rholam^{L}_{\min}\int \E\bigg[\bigg(I\{t\in B_n\}h(t)\prod_{i=1}^{m-1}X_t^{(i)}\bigg)^2\bigg] \nu(dt)
\cr &=& \rholam^{L}_{\min} \int_{B_n} h^2(t)\nu(dt). 
\ees
As the operator $K_W$ is bounded by Assumption A, $\rholam^{L}_{\max}$ and $\rholam^{L}_{\min}$ 
are both finite, so that the inequalities still hold as $B_n\to\calT$.} 
\end{proof}

\bibliographystyle{plainnat}
\bibliography{SPA-Nonlinear-Revision-July-2020.bib}
\end{document}